\documentclass[twoside,11pt]{article}

\usepackage{imon_sty}
\usepackage{blindtext}
\usepackage{microtype}
\usepackage{subfigure}
\usepackage{booktabs} 

\usepackage{hyperref}

\newcommand{\pow}[1]{^{(#1)}}

\usepackage{nicefrac}
\usepackage{amsmath}
\usepackage{mathtools}
\usepackage{bbm}
\usepackage[capitalize,noabbrev]{cleveref}
\renewcommand{\Vec}[1]{\text{Vec}\lp #1\rp}

\usepackage[textsize=tiny]{todonotes}

\numberwithin{equation}{section} 
\newcommand{\lp}{\left(}
\newcommand{\rp}{\right)}
\newcommand{\lc}{\left\{}
\newcommand{\rc}{\right\}}
\newcommand{\lb}{\left[}
\newcommand{\rb}{\right]}
\newcommand{\lv}{\left|}
\newcommand{\rv}{\right|}
\newcommand{\lV}{\left\|}
\newcommand{\rV}{\right\|}

\newcommand{\History}{\mathcal{H}}
\newcommand{\history}{\hbar}

\newcommand{\Nbb}{\mathbb{N}}

\newcommand{\Pbb}{\mathbb{P}}

\newcommand{\Rbb}{\mathbb{R}}

\newcommand{\Xbb}{\mathbb{X}}

\newcommand{\Fcal}{\mathcal{F}}

\newcommand{\Mcal}{\mathcal{M}}
\newcommand{\Ncal}{\mathcal{N}}

\newcommand{\Pcal}{\mathcal{P}}

\newcommand{\Rcal}{\mathcal{R}}
\newcommand{\Tcal}{\mathcal{T}}

\DeclareMathAlphabet{\mathdutchcal}{U}{dutchcal}{m}{n}

\newcommand{\pcal}{\mathdutchcal{p}}

\newcommand{\statespace}{\{1,\dots, d\}}

\newcommand{\indexeddata}{\left\{(X_0,a_0),\dots,(X_n,a_n)\right\}}

\newcommand{\naturalset}{\Nbb}

\newcommand\numberthis{\addtocounter{equation}{1}\tag{\theequation}}

\newcommand{\prob}{\Pbb}
\newcommand{\expec}{\mathbb{E}}
\newcommand{\probl}{\prob}
\newcommand{\initialD}{D_0}

\newcommand{\indicator}{\mathbbm{1}}

\newcommand{\nrm}[1]{\left\Vert #1 \right\Vert}

\newcommand{\beq}{\begin{eqnarray*}}
\newcommand{\eeq}{\end{eqnarray*}}
\newcommand{\beqn}{\begin{eqnarray}}
\newcommand{\eeqn}{\end{eqnarray}}
\newcommand{\ben}{\begin{enumerate}}
\newcommand{\een}{\end{enumerate}}
\newcommand{\bit}{\begin{itemize}}
\newcommand{\eit}{\end{itemize}}
\newcommand{\hide}[1]{}

\newcommand{\gn}{\, | \,}

\newcommand{\mmrisk}{\mathcal{R}_{m}}

\newcommand{\eps}{\varepsilon}
\newcommand{\vertiii}[1]{{\left\vert\kern-0.25ex\left\vert\kern-0.25ex\left\vert #1 
    \right\vert\kern-0.25ex\right\vert\kern-0.25ex\right\vert}}

\renewcommand{\epsilon}{\eps}

\usepackage{lastpage}
\jmlrheading{23}{2022}{1-\pageref{LastPage}}{1/21; Revised 5/22}{9/22}{21-0000}{Author One and Author Two}

\ShortHeadings{CLT of Controlled Markov Chains}{CLT of Controlled Markov Chains}
\firstpageno{1}

\begin{document}

\title{Central Limit Theorems for Transition Probabilities of Controlled Markov
Chains}

\author{\name Ziwei Su \email ziwei.su@northwestern.edu \\
       \addr Department of Industrial Engineering and Management Sciences\\
       Northwestern University\\
       Evanston, IL 60208, USA
       \AND
       \name Imon Banerjee \email imon.banerjee@northwestern.edu \\
        \addr Department of Industrial Engineering and Management Sciences\\
       Northwestern University\\
       Evanston, IL 60208, USA
       \AND 
        \name Diego Klabjan \email d-klabjan@northwestern.edu \\
        \addr Department of Industrial Engineering and Management Sciences\\
       Northwestern University\\
       Evanston, IL 60208, USA}

\editor{}

\maketitle

\begin{abstract}
We develop a central limit theorem (CLT) for a non-parametric estimator of the transition matrices in controlled Markov chains (CMCs) with finite state-action spaces. 
Our results establish precise conditions on the logging policy under which the estimator is asymptotically normal, and reveal settings in which no CLT can exist. 
We then build on it to derive CLTs for the value, Q-, and advantage functions of any stationary stochastic policy, including the optimal policy recovered from the estimated model.  
Goodness-of-fit tests are derived as a corollary, which enable to test whether the logged data is stochastic.
These results provide new statistical tools for offline policy evaluation and optimal policy recovery, and enable hypothesis tests for transition probabilities.
\end{abstract}


\section{Introduction}
A discrete-time stochastic process $\{(X_i, a_i)\}_{i \geq 0}$ is called a \textbf{controlled Markov chain} \citep{borkar1991topics} if the next ``state" $X_{i+1}$ depends only on the current state $X_i$ and the current (not necessarily Markovian) ``control" $a_i$, where actions $a_i$ depend only on the information available up to time $i$. 
Conditioned on $a_i$, the state sequence $\{X_i\}_{i=0}^n$ follows a Markov transition kernel (see Definition \ref{def:CMC} for the formal definition). 
We study the asymptotic distribution of the transition probabilities of finite state-action non-Markovian controlled Markov chains.

In general, controlled Markov chains with Markovian controls (where $a_i$ depends only on $X_i$) can be used to model time-homogeneous data (like i.i.d., Markovian), time-inhomogeneous data (like i.n.i.d., time-inhomogeneous Markovian \citep{dolgopyat_local_2023,merlevede_local_2022}, Markov decision process \citep{hernandez-lerma_recurrence_1991}), etc. 
Such models have wide applicability in reinforcement learning \citep{sutton2018reinforcement}, system stabilization \citep{yu_online_2023}, or system identification \citep{ljung_system_1999,mania_active_2020}. 

Beyond that, certain categories of adversarial Markov games \citep{wang_foundation_2024}, reward machines \citep{icarte_using_2018}, and minimum entropy explorations \citep{mutti_importance_2022} are known to induce Markovian state transitions with non-Markovian controls.  
Therefore, there is an acute need to sharply characterize the transition dynamics of controlled Markovian systems in the presence of non-Markovian controls. 

A particularly relevant downstream task in modern machine learning is offline (batch) reinforcement learning (RL).
Offline RL often begins with a fixed dataset of trajectories collected by some unknown logging policy (behavior policy), with no further control over data collection.
A fundamental challenge in this setting is to accurately estimate the Markov transition dynamics of the environment from these logged data. 
The estimated transition probabilities are then used to evaluate policies or to find an optimal policy for the given dynamics (see Section \ref{sec:RL} for more details).


There have been some recent developments towards understanding statistical properties of the marginal distributions of time-inhomogeneous Markov chains (see \cite{dolgopyat_local_2023} for details). However, despite the widespread prevalence of using the count-based estimator for transition probabilities in model-based RL \citep{mannor2005empirical, li2022settling, zhu2024uncertainty}, the statistical properties of this estimator are poorly understood.  
When the state space and the action space are finite, the simplest and most natural approach is to use the non-parametric count-based empirical estimator (model) of the transition probabilities (see \cref{eq:estimator}). 
We show that this estimator is essentially the maximum likelihood estimate of the transition probabilities (see Proposition \ref{prop:mle}).

One recent work \citep{banerjee2025off} establishes probably approximately correct (PAC) \citep{valiant_theory_1984} guarantees on the estimation error. Their analysis underpins the theoretical nature of PAC learning which does not always align with practice. In particular, it leaves open the important question of the exact asymptotic behavior of the estimators. This gap in theory limits our ability to evaluate estimation error, construct confidence intervals, or perform rigorous hypothesis tests for transition probabilities.

This paper addresses this gap by developing the first asymptotic normality theory for count-based transition estimators in controlled Markov chains with finite states and actions. 
We answer the following question:
\begin{quote}
    \centering
    \textit{Under what conditions does there exist a scaling such that the scaled estimation error of the empirical transition estimates obeys a CLT, and when does an asymptotic distribution fail to exist?}
\end{quote}

We show that the answer depends crucially on properties of the logging policy (such as sufficient coverage of all state-action pairs). 
In fact, under some mild recurrence and mixing conditions, we prove that the empirical transition probabilities converge in distribution to a Gaussian under self-normalized scaling equaling the square root of the visitation count of each state–action pair. 
Our analysis reveals that the estimated transition probabilities under this scaling have a multinomial behavior in the limit and are asymptotically independent across state-action pairs (see Theorem \ref{thm:eta-clt}).

Our results generalize classical CLTs for transition probabilities of ergodic Markov chains \citep{billingsley1961statistical} to non-ergodic, non-stationary stochastic processes, being a first of its kind. Furthermore, we also identify scenarios in which a CLT cannot be established for any estimator of the transition probabilities. 

To demonstrate the applicability of our results in downstream tasks such as offline policy evaluation (OPE) and optimal policy recovery (OPR), we then derive asymptotic guarantees for the value ($V$), action-value ($Q$), or advantage ($A$) functions for any stationary target policy by solving the Bellman equation with the plug-in transition estimator.
In particular, by analyzing the value of the optimal policy calculated from the estimated transition probabilities, we establish a CLT for the value function of the optimal policy under the true transition kernel. 

Our results enable, for example, confidence intervals on the value of the optimal policy and high-probability guarantees for recovering the true optimal policy, providing a principled way for OPE and OPR using asymptotic statistics. 

\paragraph{Contributions.} 
The paper develops an asymptotic theory for non-parametric estimation of transition probabilities in finite controlled Markov chains (CMCs) under general, possibly non-stationary and history-dependent logging policies. Our main contributions are as follows.
\begin{enumerate}
    \item \textbf{A central limit theorem for count-based transition estimators in CMCs.} We establish the first CLT for the classical empirical estimator of controlled transition kernels (Theorem~\ref{thm:eta-clt}). The result holds under mild assumptions on the return times of state–action pairs and weak $\eta$-mixing of the joint state–action process. 
    The covariance structure is multinomial and asymptotically independent across state–action pairs. 
    Based on the CLT, we derive a chi-square goodness-of-fit test for transition probabilities (Proposition~\ref{prop:gof}).
    \item \textbf{A characterization of when a CLT cannot hold.} We show that if the logging policy fails to visit certain state–action pairs infinitely often asymptotically, then no estimator can satisfy any CLT under any diverging scaling (Proposition~\ref{thm:no-clt}). This result draws clear regions of testable and untestable dynamics in offline RL data.
    \item \textbf{CLTs for value, Q-, and advantage functions under arbitrary stationary target policies.} Using the transition-kernel CLT, we obtain joint asymptotic normality for estimates of value, Q-, and advantage functions (Theorem~\ref{thm:VQA}), as well as the value of the optimal policy computed from the estimated model (Theorem~\ref{thm:optpol}). 
    These results provide a unified asymptotic framework for offline policy evaluation and optimal policy recovery.
\end{enumerate}
Beyond high-level contributions outlined above, the paper introduces two core technical innovations that may be of independent interest.
\begin{enumerate}
    \item \textbf{Relaxation of uniform bounded return times to sublinear growth.} 
    Prior work assumes finite expected return times for every state--action pair. This includes both regularity assumptions for Markov chains \citep[Chapter 11]{meyn_markov_2012} and its controlled counterpart \citep{banerjee2025off}. In contrast, our main theorem allows return times to grow sublinearly with the number of returns. This is enabled by \Cref{prop:count-lower-bound}, which gives a novel martingale-based lower bound on the expected number of visits to each state-control pair, even when expected return times grow.
    \item \textbf{A policy-dependent recurrence classification for CMCs.} Classical recurrence classification in Markov chains does not extend to controlled processes. 
    We identify the CMC analogue:
    \begin{itemize}
        \item a positive-recurrent-like regime guaranteeing CLTs (Proposition~\ref{prop:count-lower-bound}),
        \item a transient-like regime where no CLT is possible (Proposition~\ref{thm:no-clt}),
        \item and an intermediate “null-recurrent” regime that remains theoretically delicate.
    \end{itemize}
    This provides, to the best of our knowledge, the first testability boundary for asymptotic inference in non-stationary CMCs.
\end{enumerate}
\paragraph{Outline.}
The remainder of the paper is organized as follows.
Section~\ref{sec:lit-review} concludes the introduction with a review of the literature. 
Section~\ref{sec:notation} introduces the notions of hitting and return times, mixing, and ergodic occupation measure, and formally introduces our assumptions. 
Section~\ref{sec:CLT} introduces the CLT for transition probabilities, a goodness-of-fit test for transition probabilities, and scenarios where a CLT does not hold. 
Section~\ref{sec:RL} applies our main results to derive CLTs for the value, Q-, advantage functions and the optimal policy value. 

\subsection{Literature Review}
\label{sec:lit-review}
We divide the literature review into two parts. 
In the first part, we place our work in the context of the existing literature on non-parametric estimation for stochastic processes.
In the second part, we place our work in the context of reinforcement learning.

\paragraph{Non-parametric Estimation.} Non-parametric estimation and limit theory for Markov chains are well documented in the literature.
\cite{billingsley1961statistical} develops empirical transition estimators and asymptotic guarantees for finite ergodic chains, while \cite{yakowitz1979nonparametric} extends these ideas to infinite or continuous state spaces.
For time-homogeneous chains, \cite{geyer1998markov} proves laws of large numbers (LLNs) and \cite{jones2004markov} establishes corresponding CLTs under mixing conditions.
More recent work shifts attention to minimax sample complexity: \cite{wolfer_estimating_2019,wolfer2021statistical} derives optimal rates for ergodic chains, and \cite{banerjee2025off} does so for discrete-time, finite state-action controlled Markov chains.
Beyond transition-kernel estimation, classical results by \cite{dobrushin1956central,dobrushin1956centralII,rosenblatt1963some,rosenblatt1964some} provide LLNs and CLTs for normalized sample means in time-inhomogeneous settings, but leaves open the question of estimating the transition matrix itself.

Despite this progress, statistical inference for time-inhomogeneous controlled Markov chains—particularly when controls are stochastic and the state–action process is non-Markovian—remains unexplored.
The present paper fills this gap by establishing the first CLT for the count-based, non-parametric estimator of the transition kernel in such controlled settings, together with its implications for model-based offline reinforcement learning.

\paragraph{Reinforcement Learning.} In model-based offline (batch) reinforcement learning \citep{levine2020offline, kidambi2020morel, yu_mopo_2020}, it is important to learn or evaluate policies solely from logged trajectories without additional exploration.
Applications range from clinical decision making—where logged physician actions serve as controls \citep{shortreed2011informing, liu2020reinforcement, yu2021reinforcement, chen2021probabilistic}—to service-operations settings such as patient flow and inventory routing \citep{armony2015patient}.

Parallel works on OPE include importance sampling \citep{precup2000eligibility}, bootstrap \citep{hanna2017bootstrapping, hao2021bootstrapping}, doubly-robust \citep{pmlr-v48-jiang16, thomas2016data} and concentration-inequality-based \citep{thomas2015high} estimators that provide high-confidence bounds on a target policy’s value, while research on OPR \citep{antos2008learning, munos2008finite, chen2019information, zanette2021exponential, shi2022pessimistic} examines when a policy optimized on the learned model is near-optimal. 
Most existing guarantees are finite-sample or PAC in nature and assume either sufficient state–action coverage or Markovian behavior policies.
From a minimax sample-complexity perspective, \cite{li2022settling, yan2022model} establish sharp optimal rates for discounted and finite-horizon settings with stationary logging policies, and \cite{banerjee2025off} extends this analysis to discrete-time, finite state–action controlled Markov chains with stochastic, history-dependent logging.

Recent work by \cite{zhu2024uncertainty} studies statistical uncertainty quantification in reinforcement learning and establishes CLTs for value and Q-functions of the optimal policy.
In particular, their Theorem 1 builds on classical Markov chain CLTs \citep{trevezas2009variance} by viewing each state–action pair as a single composite state and thus assuming stationary, irreducible logging data.
While they also employ count-based estimators for transitions and empirical estimators for rewards, the asymptotic normality of these estimators is taken as an assumption rather than derived, and their framework does not accommodate non-stationary logging policies.
In contrast, we derive from first principles a CLT for the count-based estimator of controlled transition probabilities themselves under general logging policies, including fully non-stationary, stochastic, and history-dependent policies. 
By establishing transition-level asymptotics without stationarity, our results provide a fundamentally different inference foundation for OPE and OPR. 

\section{Notations, Background and Assumptions}
\label{sec:notation}
\subsection{Notations}  
We assume that the state process $\{X_i\}_{i \geq 0}$ takes values in a finite state space $\mathcal S$ with $|\mathcal S| = d$, and the action process $\{a_i\}_{i \geq 0}$ takes values in a finite action space $\mathcal A$ with $|\mathcal A| = k$.
Throughout this paper, all random variables are defined on a filtered probability space $(\Omega, \mathcal F, \mathbb F, \mathbb P)$, where $\mathbb F := \{\mathcal F_j\}_{j \geq 0}$, and $\mathcal F_j$ is a filtration with $\mathcal F_j \subset \mathcal F$ defined as below.
Let the history from time $p$ to $j$ be denoted by $\History_p^{j} := \left\{\lp X_i, a_i\rp\right\}_{i=p}^{j}$, and define the $\sigma$-algebra generated by $\History_0^{j}$ as $\mathcal F_j$.
For readability and to avoid over-notation, we use $\mathcal F_j$ in measure-theoretic contexts as in \cref{eq:ass1} and $\History_0^{j}$ when writing explicit conditioning in transition probabilities as in \cref{eq:CMC}.
 
Following \cite{borkar1991topics}, we introduce the following definition of a controlled Markov chain.

\begin{definition}
\label{def:CMC}
A controlled Markov chain (CMC) is an $\mathbb F$-adapted pair process $\left\{\lp X_i, a_i\rp\right\}_{i\geq 0}$
whose transition probabilities are given by
\begin{align}
\label{eq:CMC}
     M_{s,t}^{(l)} := \mathbb P\lp X_{i+1} = s \;\mid\; X_i = t, a_i = l \rp, \quad \text{ for all time points $i$}
\end{align}
and which satisfies the Markov property
\begin{align}
\mathbb P\lp X_{i+1} = s_{i+1} \;\mid\; X_i = s_i, a_i = l_i \rp = \mathbb P \lp X_{i+1} = s_{i+1} \;\mid\; \mathcal H_0^i = \history_0^i \rp, \nonumber
\end{align}
where $\history_0^i := \{X_0=s_0, a_0=l_0, \ldots, X_i=s_i, a_i=l_i\}$ is the sample history up to time $i$.
A Markov decision process (MDP) is a CMC with a reward process $\{r_i\}_{i \geq 0}$.
\end{definition}
The action sequence $\{a_i\}$ is adapted to $\{\mathcal F_i\}$ for all $i \geq 0$ and is allowed to be non-Markovian and time-inhomogeneous. 
For example, when $\{a_i\}$ is deterministic, $\{X_i\}$ forms a time-inhomogeneous Markov chain whose transition matrix at time $i$ is determined by $a_i$. 
Example~\ref{exp:return-time-growth} illustrates such a chain.
Conversely, when $a_i$ depends only on $X_i$, the pair process $\{(X_i,a_i)\}$ forms a time-homogeneous Markov chain.

For each state $s \in \mathcal S$, let $M_s$ represent the matrix
\begin{align*}
\begin{bmatrix}
     M_{s,1} \pow 1,\dots,M_{s,1} \pow k\\
     M_{s,2} \pow 1,\dots,M_{s,2} \pow k\\
     \vdots \\
     M_{s,d-1} \pow 1,\dots,M_{s,d-1} \pow k\\
     M_{s,d} \pow 1,\dots,M_{s,d} \pow k
\end{bmatrix}
.\numberthis \label{eq:M_s}
\end{align*}
As $d$ and $k$ are finite, the collection of such matrices is finite.

Let $N_s\pow l$ be the number of visits to the (state, action) pair $(s,l)$, and $N_{s,t}\pow l$ be the number of transitions from state $s$ to state $t$ under $l$. Formally, with $\indicator[\cdot]$ as indicator function, 
\begin{align*}
    N_s \pow l & :=\sum_{i=0}^{n-1} \indicator[X_i=s,a_i=l], & N_{s,t} \pow l :=\sum_{i=0}^{n-1} \indicator[X_i=s,X_{i+1}=t,a_{i}=l].  
\end{align*}
These quantities depend on the size of the offline data $n$.
We omit this dependency when there is no scope of confusion, otherwise we write $N_s \pow l(n)$ and $N_{s,t} \pow l(n)$.

Our count-based estimator denoted $\hat M_{s,t}\pow l$ of the transition probability from state $s$ to $t$ conditioned on action $l$, is defined as
\begin{align}
\label{eq:estimator}
    \Hat{M}_{s,t}^{(l)} := \frac{ N_{s,t}^{(l)} }{ N_s^{(l)}}.
\end{align}
It follows from first principles that the random variable $\Hat{M}_{s,t}^{(l)}$ is the maximum likelihood estimator of the transition probability $M_{s,t} \pow l$. The following proposition formalizes this and we provide a proof in Appendix~\ref{sec:prf-mle}.
\begin{proposition}
\label{prop:mle}
Random variable $\Hat{M}_{s,t}^{(l)}$ is the maximum likelihood estimator of the transition probability $M_{s,t} \pow l$.
\end{proposition}

Understanding the asymptotic distribution of maximum likelihood estimators is a classic topic in statistics, with a long history in the analysis of i.i.d. and stationary data \citep{cramer1946mathematical, billingsley1961statistical}. 
Our goal here is to establish the asymptotic normality of the scaled estimation error
\[
b_n\lp \hat M_{s,t}\pow l (n) -  M_{s,t}\pow l \rp,
\]
where $b_n$ is a (possibly random) scaling factor.

In our setting, however, the challenge is considerably greater: the data is non-stationary, and this dependence structure makes determining the correct scaling in any CLT argument particularly delicate.
Our analysis reveals that the correct scaling factor is random and, in fact, equal to the visitation count $N_s \pow l(n)$.

Establishing such a CLT requires both sufficient coverage of the relevant state–action pairs and suitably weak temporal dependence. 
To formalize these conditions, we next introduce the necessary background concepts and assumptions.
\subsection{Assumptions and Implications}
\paragraph{Hitting and Return Times.}
For standard (uncontrolled) Markov chains, the notion of \emph{recurrence} is rigorously characterized using hitting and return times. In particular, a Markov chain can be classified as either positive recurrent, null recurrent, or transient—an exhaustive and mutually exclusive trichotomy. 
To the best of our knowledge, no canonical analogue of this classification exists for CMCs. 
We therefore define analogous notions through the return time of the state-action pairs.
We recursively define the ‘time of return’ for every state-action pair $(s,l)$ as follows.
\begin{definition}~\label{def:return-time}
The first \emph{hitting time} $(s,l)$ is defined as
\begin{align*}
&\tau_{s,l}\pow{1}\in \min\lc i > 0: X_i=s, a_i=l\rc.
\end{align*}
When $i\geq2$, the $i$-th {\emph{time to return}} (or {\emph{return time}}) of the state-action pair $(s,l)$ is recursively defined as 
\begin{align*}
  \tau_{s,l} \pow i \in \min\lc k>0:X_{\sum_{j=1}^{i-1}\tau_{s,l} \pow j+k}=s,a_{\sum_{j=1}^{i-1}\tau_{s,l} \pow j+k}=l \rc.  
\end{align*}
In these definitions, the sets on the right-hand side are singletons.
\end{definition}

The following assumption plays a role analogous to \textit{positive recurrence} in standard Markov chains, ensuring that every state-action pair $(s,l)$ is visited sufficiently often. 
\begin{assumption}~\label{ass:return-time-growth}
For all state-action pairs $(s,l) \in \mathcal{S} \times \mathcal{A}$, there exists $0< \alpha < 1$ such that
\begin{align}
\label{eq:ass1}
    \mathbb{E} \left[\tau^{(i)}_{s,l} \;\middle|\;
\mathcal{F}_{\sum_{p=1}^{i-1}\tau^{(p)}_{s,l}}\right] \leq T_i \quad \text{a.s.}, \quad T_i = O(i^{\alpha}).
\end{align}
\end{assumption}
Unlike previous assumptions in the literature (like that of regularity in Markov chains--- see Theorem 11.0.1 or Equation (11.1) in \cite{meyn_markov_2012} or Assumption 2 in \cite{banerjee2025off})---which require $T_i$ in (\ref{eq:ass1}) to be uniformly bounded for all $i$ Assumption \ref{ass:return-time-growth} allows the expected return time to grow sublinearly with $i$, which admits CLT for non-stationary CMCs in which some state–action pairs are visited increasingly rarely over time, yet still infinitely often with probability one (jointly established by Proposition~\ref{prop:count-lower-bound} and Lemma~\ref{lemma:consistency}).
This relaxation allows the (possibly non-stationary and history-dependent) logging policy to concentrate sampling on only a subset of state–action pairs rather than enforcing uniform exploration of the entire state–action space. 
Example \ref{exp:return-time-growth} illustrates such a non-stationary CMC, and Appendix~\ref{sec:inhomo-chain} discusses this example in detail.
\begin{example}[Inhomogeneous Markov chain]
\label{exp:return-time-growth}
A controlled Markov chain is said to be an inhomogeneous Markov chain if there exists a sequence of actions $l_0,l_1,\ldots \in \mathcal A$ such that for any time $i$, state $s$, sample history $\history_0^{i-1} \in (\mathcal{S}\times\mathcal{A})^i$, we have
\[
\prob\lp a_i=l_i\;\middle|\;\History_0^{i-1}=\history_0^{i-1},X_i=s \rp=1.
\]
Suppose that each action $l \in \mathcal A$ appears at least once in any window of fixed length
$T_i = O(i^\alpha)$ between successive returns to any state-action pair $(s,l)$, and that all
transition probabilities are strictly positive (i.e.,
$M^{(l)}_{s,t} \ge M_{\min} > 0$ for all $s,t,l$). 
We show in Appendix~\ref{sec:inhomo-chain} that, for this process, the probability of not returning to $(s,l)$ within $k$ steps conditional on the
past is bounded by
\[
\mathbb P \lp\tau \pow i_{s,l} > k \;\middle|\;
\mathcal{F}_{\sum_{p=1}^{i-1}\tau^{(p)}_{s,l}}\rp
\leq (1-M_{\min})^{\lfloor k/T_i \rfloor -1},
\]
which implies
\[
\mathbb{E} \left[\tau^{(i)}_{s,l} \;\middle|\;
\mathcal{F}_{\sum_{p=1}^{i-1}\tau^{(p)}_{s,l}}\right] = \sum_{k=1}^\infty \mathbb P \lp\tau \pow i_{s,l} > k \;\middle|\;
\mathcal{F}_{\sum_{p=1}^{i-1}\tau^{(p)}_{s,l}}\rp = O(i^\alpha).
\]
Hence Assumption~\ref{ass:return-time-growth} holds.
The full details are formalized in Proposition~\ref{prop:inhomogenous-mc} in Appendix~\ref{sec:inhomo-chain}.
\end{example}

The following proposition shows that Assumption~\ref{ass:return-time-growth} guarantees a minimum visitation frequency for all state-action pairs.
\begin{proposition}
    \label{prop:count-lower-bound}
    For controlled Markov chain that satisfies Assumption~\ref{ass:return-time-growth}, we have
    \begin{align*}
        \mathbb{E}\lb N_s \pow l (n) \rb = \Omega\lp n^{\frac{1}{1+\alpha}}\rp.
    \end{align*}
\end{proposition}
The proof of this proposition can be found in Appendix~\ref{sec:proof-count-lower-bound}.

\paragraph{Mixing Coefficients.}
Following \cite{kontorovich_concentration_2008}, we define the weak and uniform mixing coefficients to quantify temporal dependence in the paired process of states and actions. 
Intuitively, these coefficients measure how much the future of the process can change if we alter the state–action pair at a single point of the past. 

For any $j \leq n$, $j, n \in \mathbb N$, sample history $\history_0^{i-1} \in (\mathcal{S}\times\mathcal{A})^i$, let $\Tcal\subseteq(\mathcal{S}\times\mathcal{A})^{n-j+1}$, $s_1,s_2 \in \mathcal{S}$, and $l_1,l_2 \in \mathcal{A}$. 
For any $\history_0^{i-1}$ and $i<j$, we define
\begin{align*}
    \eta_{i,j}(\Tcal,s_1,s_2,l_1,l_2,\history_0^{i-1}, n) & := \left|\prob\lp\lp X_j, a_j,\dots, X_n, a_n\rp\in \Tcal|X_i=s_1,a_i=l_1,\History_0^{i-1}=\history_0^{i-1}\rp\right.\\
    & \left.\qquad -\prob\lp\lp X_j, a_j,\dots, X_n, a_n\rp\in\Tcal|X_i=s_2,a_i=l_2,\History_0^{i-1}=\history_0^{i-1}\rp\right|.
\end{align*}
Then the \emph{weak mixing} ($\bar \eta$-mixing) coefficient $\bar\eta_{i,j}(n)$ is defined as
\begin{align*}~\label{def:weak-mixing}
    \Bar{\eta}_{i,j}:= \underset{\color{black}\substack{\Tcal,s_1,s_2,l_1,l_2,\history_0^{i-1},\\\prob\lp X_i=s_1,a_i=l_1,\History_0^{i-1}=\history_0^{i-1}\rp>0,\\\prob\lp  X_i=s_2,a_i=l_2,\History_0^{i-1}=\history_0^{i-1}\rp>0}}{\sup} \eta_{i,j}(\Tcal,s_1,s_2,l_1,l_2,\history_0^{i-1}, n).\numberthis
\end{align*}
We impose the following boundedness condition on the cumulative decay of $\bar \eta$-mixing coefficients of the paired process $\{ (X_i, a_i)\}$. 
\begin{assumption}
\label{ass:eta-mix}
    There exists a constant $C_\Delta>1$ independent of $n$ such that, 
\begin{align*}
    \|\Delta_n\|:=\underset{1\leq i\leq n}{\max} \left(1+ \Bar{\eta}_{i,i+1}(n)+ \Bar{\eta}_{i,i+2}(n)+\dots \Bar{\eta}_{i,n}(n)\right)\leq C_\Delta.
\end{align*}
\end{assumption}
This assumption controls the cumulative temporal dependence of the paired process and, in particular, it helps bound the deviation of $N_s \pow l$ from its expectation $\mathbb{E}\left[N_s \pow l\right]$, a key step in establishing CLT.

However, verifying this condition directly is typically difficult; it requires bounding joint dependencies between all states and actions over the entire trajectory, which may be analytically intractable.

To make the condition more interpretable and easier to verify, we next show how it can be reduced to separate assumptions on the mixing of the state process and the action process. 
This decomposition isolates two distinct sources of dependence, allowing each to be analyzed independently.

\paragraph{Reduction of Mixing Coefficients.} 

We begin by defining the $\gamma$-mixing coefficients  $\gamma_{p,j,i}$ for actions as the following total variation distance 

  \begin{align*}
    \gamma_{p,j,i} := 
    \sup_{\substack{s_p, \history_{i+j}^{p-1}, \history_0^i\\\prob\lp X_p=s_p,\History_{i+j}^{p-1}=\history_{i+j}^{p-1},\History_0^i=\history_{0}^{i}\rp>0}} 
    & \Big\lVert 
        \probl\lp 
            a_p \Big| X_p = s_p, \History_{i+j}^{p-1} = \history_{i+j}^{p-1}, \History_0^i = \history_{0}^i
        \rp \\&\qquad \qquad -  \probl\lp  
            a_p \Big| X_p = s_p, \History_{i+j}^{p-1} = \history_{i+j}^{p-1}
        \rp
    \Big\rVert_{TV}.
\end{align*}    
The total variation norm here captures the worst-case sensitivity of the future action distribution to distant past information, thereby quantifying the degree of temporal dependence induced by the logging policy.

We now impose a summability condition on these $\gamma$-mixing coefficients.
This condition requires that, as the time lag increases, the effect of earlier histories on future actions decays sufficiently fast.
\begin{assumption}[Mixing of Actions]~\label{assume:control-mixing} There exists a constant $C\geq 0$ such that
\begin{align*}
    \sup_{i \geq 1}\sum_{j=1}^{{\color{black} \infty}} \sum_{p=i+j+1}^{{\color{black} \infty}} \gamma_{p,j,i}\leq \frac{C}{2}.
\end{align*}

\end{assumption}
In Markovian settings, when the sequence of actions $a_p$ depends only on $X_p$,
$
 \gamma_{p,j,i}=0
$
for all $p,j,i$. 
In such a case, Assumption~\ref{assume:control-mixing} is satisfied with $C=0$. 
This extends to the case where $a_p$ depends on $k \geq 2$ many past time points. 
If $a_p$ depends only on $X_{p-k+1}, a_{p-k+1}, \dots,$ $X_{p-1}, a_{p-1}, X_p$, then Assumption~\ref{assume:control-mixing} is satisfied with $C=(k-1)(k-2)$. 

With the definitions of $\gamma$-mixing coefficients for actions in hand, we introduce counterparts for the state process, which capture how the conditional distribution of future states depends on the current state–action pair.
This extends the classical Dobrushin coefficients \citep{dobrushin1956central,dobrushin1956centralII,mukhamedov2013dobrushin} for inhomogenous Markov chains to the realm of controlled Markov chains. 

For all integers $j\geq i$, we define the mixing coefficient 
\begin{align*}
& \bar \theta_{i,j}:= \underset{\substack{s_1,s_2\in\mathcal S,l_1,l_2\in \mathcal A, (s_1,l_1)\neq(s_2,l_2)\\ \prob(X_i=s_1,a_i=l_1)>0,\\ \prob(X_i=s_2,a_i=l_2)>0  }}{\sup}  \left\| \probl\lp X_j|X_i=s_1,a_i=l_1\rp-\probl\lp X_j|X_i=s_2,a_i=l_2\rp \right\|_{TV}.\numberthis\label{eq:def-theta}
\end{align*}
Intuitively, $\bar \theta_{i,j}$ measures the rate at which the \emph{state process} ``forgets'' its starting point when no condition on controls are given.

We now impose a summability condition on these $\bar \theta$-mixing coefficients.
This condition ensures that the influence of an initial state–action pair on future states decays sufficiently fast.
\begin{assumption}[Mixing of States]~\label{assume:chain-mix} There exists a constant $C_{\theta}\geq 0$ such that
\[
\sup_{i \geq 1}\sum_{j=i+1}^{{ \infty}} \bar \theta_{i,j}\leq C_{\theta}. 
\]
\end{assumption}

Example \ref{exp:ns-markov-controls} illustrates a non-stationary CMC that satisfies Assumptions~\ref{assume:control-mixing} and \ref{assume:chain-mix}, and Appendix~\ref{sec:non-stat-chain} discusses this example in detail.
\begin{example}[Non-stationary Markov controls]
\label{exp:ns-markov-controls}
A controlled Markov chain is said to have non-stationary Markov controls if for any time $i$, state $s$, action $l$, and sample history $\history_0^{i-1}$, we have
\[
\probl\lp a_i=l|X_i=s,\History_0^{i-1}=\history_0^{i-1}\rp=\probl\lp a_i=l|X_i =s\rp.
\]
We observe that the law of the action sequence can depend on the time $i$. 
Let $P_{s,l}^{(i)} := \prob(a_i=l|X_i=s)$.
We can write the transition probability of the state-action pair as
\begin{align*}
&\prob\lp X_i=t,a_i=l'|X_{i-1}=s,a_{i-1}=l\rp\\
= \ &\prob\lp X_i=t|X_{i-1}=s,a_{i-1}=l\rp\prob(a_i=l'|X_i=t)\\
= \ & M_{s,t}^{(l)}\cdot P_{t,l'}^{(i)}.
\end{align*}
It is straightforward to see that the state-action pair is a time-inhomogeneous Markov chain with transition probabilities given by $M_{s,t}^{(l)} P_{s,l'}^{(i)}$. 

As $a_i$ depends only on the \emph{current} state $X_i$ (and not on deeper history), it has no long-range dependence. Thus,
    \begin{align*}
\gamma_{p,j,i} & = 
    \sup_{s_p, \history_{i+j}^{p-1}, \history_0^i} 
    \Big\lVert 
        \probl\lp 
            a_p \Big| X_p = s_p, \History_{i+j}^{p-1} = \history_{i+j}^{p-1}, \History_0^i = \history_{0}^i
        \rp - 
        \probl\lp  
            a_p \Big| X_p = s_p, \History_{i+j}^{p-1} = \history_{i+j}^{p-1}
        \rp
    \Big\rVert_{TV} \\
    & \equiv 0.
\end{align*}
Therefore, Assumption~\ref{assume:control-mixing} holds for all controlled Markov chains with non-stationary controls with $C=0$.

Suppose that all transition probabilities are strictly positive (i.e., $M^{(l)}_{s,t} \ge M_{\min} > 0$ for all $s,t,l$). 
Then, by \citet[Lemma 12]{banerjee2025off}, Assumption~\ref{assume:chain-mix} holds with $C_\theta = 1/(d M_{min})$.
\end{example}
Neither of Assumptions \ref{assume:control-mixing} and \ref{assume:chain-mix} imply the other, as the following counter-examples illustrate.
\begin{enumerate}
    \item Let $(X_i,a_i)$ be an inhomogeneous Markov chain for which  
    $
    \sup_{1\leq i\leq \infty}\sum_{j=i+1}^\infty \bar \theta_{i,j}=\infty.
    $
   However, since the actions are deterministic, every inhomogenous Markov chain satisfies Assumption \ref{assume:control-mixing}. We prove this fact formally in \Cref{prop:inhomogenous-mc}. Therefore, this chain satisfies Assumption \ref{assume:control-mixing} but not Assumption \ref{assume:chain-mix}.
    \item For the second counter-example consider a controlled Markov chain $(X_i,a_i)$ where the $a_i$'s do not satisfy Assumption \ref{assume:control-mixing}. Let $X_i$ be independent draws from a uniform distribution over $\mathcal S$. It is easily seen that $\bar\theta_{i,j}=0$ for this example. Therefore, this chain satisfies Assumption \ref{assume:chain-mix} but not \ref{assume:control-mixing}.
\end{enumerate}

Finally, we observe that the previous assumptions on the states and actions imply the summability of the weak mixing coefficients. 
Due to \citet[Lemma 3]{banerjee2025off}, for any controlled Markov chain that satisfies Assumptions \ref{assume:control-mixing} and \ref{assume:chain-mix},
$
    \lVert\Delta_n\rVert\leq C+C_{\theta}+1.
$

\paragraph{Ergodic Occupation Measure.}
Next, we define the \emph{ergodic occupation measure} in the context of CMCs. 
\begin{definition}[Ergodic Occupation Measure]\label{def:semi-ergodic}
For every $(s,l) \in \mathcal S \times \mathcal A$, we define
\begin{align*}
p_s\pow l := \lim_{n\rightarrow\infty} \frac{\sum_{i=0}^{n-1} \prob\lp X_i=s,a_i=l \rp}{n},
\end{align*}
whenever the limit exists. 
The limit $p_s\pow l$, if it exists for all $(s,l) \in \mathcal S \times \mathcal A$, is called the ergodic occupation measure of the CMC.
\end{definition}

Observe that for stationary CMCs, $p_s\pow l=\prob(X_0=s,a_0=l)$.  
Some usual processes, like episodic CMCs, are not stationary, but an ergodic occupation measure exists. 
We make the following assumption essential for the CLTs for value, Q-, and advantage functions.
\begin{assumption}
\label{ass:semi-ergodic}
    The CMC has the ergodic occupation measure, and $\min_{s,l} p_s \pow l = T>0$.
\end{assumption}
The previous assumption can be relaxed with appropriate assumptions on the return time of $(X_i,a_i)$ which would translate into an assumption about the logging policy. 
\section{CLT for Controlled Markov Chains}
\label{sec:CLT}
In this section, we establish CLTs for count-based empirical estimates of transition probabilities in CMCs.
We begin with the ``properly scaled" CLT, which characterizes the asymptotic distribution of transition estimates under state-action pair-specific normalization, and provide a sketch of its proof, highlighting a new sampling construction that generalizes the classical approach of \citet{billingsley1961statistical}. 
We then introduce an auxiliary, $\sqrt{n}$-scaled CLT that serves as the basis for the CLTs of value, Q-, and advantage functions in Section~\ref{sec:RL}, followed by a discussion of regimes where no CLT exists. 
The section concludes with a goodness-of-fit test that illustrates the practical implications of the established asymptotic results.

\subsection{Properly Scaled CLT}
We begin with introducing additional notations.
Let $\odot$ denote the Schur or Hadamard product of two compatible vectors and let $\otimes$ denote the Kronecker product. 
Let $\Vec A$ be the vectorization of the matrix $A$ given by stacking the columns. 

Let $\mathbf{N}$ be the $dk\times d$ dimensional matrix given by
\begin{align*}
\begin{bmatrix}
     \lp\sqrt{N_1\pow l}\rp_{l=1}^k \otimes \mathbf{1}_d^\top\\
     \vdots\\
    \lp\sqrt{N_d\pow l}\rp_{l=1}^k \otimes \mathbf{1}_d^\top 
\end{bmatrix}, \text{ where } \mathbf{1}_d^\top := (1,\dots,1)^\top \in \Rbb^d.
\end{align*}
The $((s-1)k+l, j)$-th entry of $\mathbf{N}$ equals $\sqrt{N \pow l_s}$ for all $1 \leq j \leq d$.

Let $\mathbf{M}$ be the $dk\times d$ dimensional matrix given by $[M_1, M_2,\dots, M_d]^\top$ , where $M_s$ is given in \cref{eq:M_s}.
We denote its count-based estimated counterpart by $\mathbf{\hat M}$.
Let $\xi:=\lp\Vec{\mathbf{\hat M}^\top}-\Vec{\mathbf{M}^\top}\rp\odot\Vec{\mathbf{N}^\top}$. 
Elementwise, the $(s-1)dk+(l-1)d+t$-th element of $\xi$ is
$\sqrt{N_s\pow l}\lp{\hat M_{s,t}\pow l-M_{s,t}\pow l}\rp$.
We now state the CLT for count-based empirical estimates of transition probabilities in CMCs.
\begin{theorem}~\label{thm:eta-clt}
    Under Assumptions~\ref{ass:return-time-growth} and \ref{ass:eta-mix}, 
    \begin{align*}
    \xi  \overset{d}{\rightarrow} \Ncal(0,\Lambda),
    \end{align*}
    where $\Lambda$ is a covariance matrix. The elements of $\Lambda$ are given by $\lambda_{slt,s'l't'}$ which denotes the covariance between the $(s-1)dk+(l-1)d+t$-th and the $(s'-1)dk+(l'-1)d+t'$-th elements of $\xi$.
    Value $\lambda_{slt,s'l't'}$ has the expression
    \begin{align*}
    \lambda_{slt,s'l't'} = \indicator[(s,l)=(s',l')]\lp \indicator[t=t'] M_{s,t}\pow l-M_{s,t}\pow l M_{s',t'}\pow {l'}\rp.\numberthis\label{eq:correlation} 
    \end{align*}
\end{theorem}

We observe that the covariance matrix $\Lambda$ depends only on the transition probabilities
$\{M_{s,t}^{(l)}\}$ and does not depend on the existence of an ergodic occupation
measure.
Consequently, the asymptotic covariance structure is invariant to the logging policy, provided the conditions of Theorem~\ref{thm:eta-clt} hold.

Theorem~\ref{thm:eta-clt} generalizes the classical CLT for ergodic Markov chains \citep{billingsley1961statistical} to the setting of controlled and potentially non-stationary dynamics. 
Unlike standard Markov chain CLTs that rely on ergodicity or stationarity, our result accommodates both non-Markovian and time-inhomogeneous dependencies in the action sequence and transition dynamics. 
The covariance matrix $\Lambda$ mirrors that of multinomial trials, capturing the asymptotic covariance of transition counts within each state–action pair and the asymptotic independence across distinct pairs. 
In practice, $\Lambda$ can be consistently estimated by replacing $M_{s,t}\pow l$ with their empirical counterparts $\hat M_{s,t}\pow l$ in~\cref{eq:correlation}.

The main challenge in proving Theorem~\ref{thm:eta-clt} is the long-range dependence between the state and action sequences in a CMC induced by history-dependent, non-stationary policies.
This dependence breaks the Markov property of the marginal state process and invalidates classical proof techniques based on regenerative arguments \citep{doeblin1938deux, dobrushin1956central, dobrushin1956centralII, nummelin1978splitting}. 
To overcome this difficulty, we introduce an auxiliary sampling scheme that replicates the joint evolution of $(X_i,a_i)$ while decoupling temporal dependence across samples.
This construction extends the seminal sampling scheme approach of \citet{billingsley1961statistical} to the controlled and non-stationary setting. 
It provides the key technical innovation that enables the use of multinomial-type CLT arguments in the controlled regime.

We next provide a proof sketch of Theorem~\ref{thm:eta-clt}. 
The complete proof is deferred to Appendix~\ref{sec:proof-eta-clt}.

\subsection{Proof Sketch of Theorem~\ref{thm:eta-clt}} 
\begin{proof}
The proof adapts the classical CLT for ergodic, finite Markov chains \citep{billingsley1961statistical} to the controlled setting by constructing an auxiliary sampling scheme that decouples temporal dependence from the empirical transition counts. 
We outline the key steps.

\paragraph{Step 1 (Consistency).}
We first establish (see Lemma~\ref{lemma:consistency} in Appendix~\ref{sec:lemma_consistency}) that
\[
\frac{N_s\pow l}{\mathbb E \left[N_s\pow l \right]} \xrightarrow{a.s.} 1.
\]

\paragraph{Step 2 (Sampling Scheme).}
For each action $l \in \mathcal{A}$, we construct an infinite array of i.i.d. random variables that are also independent of the observed data $\indexeddata$:
\begin{align}
\Xbb^{(l)} = \left[
\begin{array}{ccccc}
X_{1,1}^{(l)} & X_{1,2}^{(l)} & \dots & X_{1,\tau}^{(l)} & \dots \\
X_{2,1}^{(l)} & X_{2,2}^{(l)} & \dots & X_{2,\tau}^{(l)} & \dots \\ 
\vdots & \vdots & \ddots & \vdots & \vdots \\
X_{d,1}^{(l)} & X_{d,2}^{(l)} & \dots & X_{d,\tau}^{(l)} & \dots
\end{array} \right].
\nonumber
\end{align}
Each element $X_{s,\tau}^{(l)}$ represents a possible next state when action $l$ is taken from state $s$, and follows the transition law
\[
\mathbb P \lp X_{s,\tau}^{(l)}=t \rp=M_{s,t}^{(l)}, \quad (s,t,\tau)\in \mathcal{S}\times\mathcal{S}\times\mathbb{N}.
\]
In addition, for every time $i\ge1$ and history $\lp \history_0^{i-1}, s_i\rp = \lp (s_0,l_0),\dots,(s_{i-1},l_{i-1}),s_i \rp \in(\mathcal{S}\times\mathcal{A})^{i}\times\mathcal{S}$, we define an independent random variable
\[
\alpha_i^{\history_0^{i-1}, s_i}\in\mathcal{A},
\]
with conditional distribution
\[
\mathbb P\lp\alpha_i^{\history_0^{i-1}, s_i}=l\rp
= \mathbb P \lp a_i=l \mid X_i=s_i, \History_0^{i-1}=\history_0^{i-1} \rp.
\]

The sampling proceeds recursively as follows. We initialize $(\tilde X_0,\tilde a_0)\overset{d}{=}(X_0,a_0)$.  
At each step $i\ge0$,
\begin{align*}
\tilde X_{i+1} = X^{(\tilde a_i)}_{\tilde X_i,\, \tilde N_{\tilde X_i}^{(i,\tilde a_i)}+1}, \qquad \tilde a_{i+1} = \alpha_{i+1}^{(\tilde X_0,\tilde a_0,\dots,\tilde X_{i+1})},
\end{align*}
where $\tilde N_{s}^{(i,l)}=\sum_{j\le i}\indicator[\tilde X_j=s,\tilde a_j=l]$ counts the number of visits to $(s,l)$ up to time $i$.  
Finally, let $\tilde N_s^{(l)}=\sum_{i=1}^n \indicator[\tilde X_i=s,\tilde a_i=l]$ denote the total number of visits under the constructed process.

Intuitively, each array $\Xbb^{(l)}$ provides an independent “stream” of next-state samples for action $l$, while the variables $\alpha_i$ govern the (possibly non-Markovian) control mechanism.  
The resulting sequence $(\tilde X_i,\tilde a_i)$ thus preserves the same joint distribution as the original process $(X_i,a_i)$ but decouples temporal dependence across samples.  
Formally, we show in Proposition~\ref{prop:mod-sampling-scheme} in Appendix~\ref{sec:prop_sampling} that $(\tilde X_0,\tilde a_0,\dots,\tilde X_n,\tilde a_n)\overset{d}{=}(X_0,a_0,\dots,X_n,a_n)$, yielding a coupling that retains the CMC law while enabling the application of multinomial CLTs.

\paragraph{Step 3 (Approximate Multinomial Representation).}
For each $(s,l)$, we define
\[
\tilde N_{s,t}\pow l = 
\sum_{\tau=1}^{\left\lfloor\mathbb E \left[N_s\pow l \right]\right\rfloor}
\indicator \!\left[\tilde X^{(l)}_{s,\tau}=t\right],
\quad
\tilde\xi_{s,l,t} =
\frac{\tilde N_{s,t}\pow l - \left\lfloor\mathbb E\left[N_s\pow l\right]\right\rfloor M_{s,t}\pow l}{\sqrt{\mathbb E\left[N_s\pow l\right]}}.
\]
Conditional on $N_s\pow l$, the vector
$\lp \tilde N_{s,1}\pow l,\dots,\tilde N_{s,d}\pow l\rp$
follows a multinomial distribution with number of trials $\mathbb E\left[N_s\pow l\right]$ and event probabilities $\lp M_{s,1}\pow l,\dots,M_{s,d}\pow l \rp$, implying that
\[
\tilde\xi \;\xrightarrow{d}\; \Ncal(0,\Lambda),
\]
where $\Lambda$ is given by~\cref{eq:correlation}.

\paragraph{Step 4 (Equivalence and Convergence).}
In this step (Lemma~\ref{lemma:tight} in Appendix~\ref{sec:lemma_tight}) we show that the difference between the empirical and auxiliary statistics vanishes in probability:
\[
\tilde\xi_{s,l,t}-\xi_{s,l,t}\xrightarrow{p}0.
\]
Combining these steps yields
$\xi \xrightarrow{d} \Ncal(0,\Lambda)$,
establishing Theorem~\ref{thm:eta-clt}.
\end{proof}

\subsection{Improperly Scaled CLT}
While Theorem \ref{thm:eta-clt} establishes asymptotic normality under state-action pair-specific normalization by $\sqrt{N_s \pow l}$, certain downstream analyses—such as aggregating transition effects or studying plug-in functionals—require a common global scaling. 
To this end, we introduce an ``improperly scaled" CLT, obtained by normalizing all transition estimates by the same factor (typically $\sqrt{n}$). 
Although this scaling is no longer correct for individual state–action pairs, it yields a unified limit that is analytically convenient and forms the basis for the CLTs of value, Q-, and advantage functions in Section~\ref{sec:RL}.

For a semi-ergodic controlled Markov chain with no absorbing states, we define $\pcal$ to be the matrix populated by the inverse of the square roots of the $p_s\pow l$ probabilities, 
\begin{align*}
\pcal = 
\begin{bmatrix}
     \lp1/\sqrt{p_1\pow l}\rp_{l=1}^k \otimes \mathbf{1}_d^\top\\
     \vdots\\
    \lp1/\sqrt{p_d\pow l}\rp_{l=1}^k \otimes \mathbf{1}_d^\top 
\end{bmatrix}
, 
\end{align*}
where
\begin{align*}
     \lp1/\sqrt{p_s\pow l}\rp_{l=1}^k :=
\begin{bmatrix}
     1/\sqrt{p_s\pow 1}\\
     \vdots\\
    1/\sqrt{p_s\pow k}
\end{bmatrix} 
\in \mathbb R^k.
\end{align*}
Using Theorem~\ref{thm:eta-clt} and Lemma~\ref{lemma:consistency}, the proof of the following corollary is immediate.
\begin{corollary}\label{cor:improper-clt}
    Let $\bar \xi :=\sqrt{n}\lp\Vec{\mathbf{\hat M}^\top}-\Vec{\mathbf{M}^\top}\rp$ be the ``improperly scaled" vector of the differences between $\mathbf{\hat M}$ and $\mathbf{M}$. 
    Consider the setting of Theorem~\ref{thm:eta-clt} and suppose that Assumptions~\ref{ass:return-time-growth}, \ref{ass:eta-mix} and \ref{ass:semi-ergodic} hold. 
    Then,  
    \begin{align*}
    \bar \xi\overset{d}{\rightarrow}\Ncal(0,\bar{\Lambda})
    \end{align*}
    where $\bar{\Lambda}$ is the improperly scaled covariance matrix. The elements of $\bar{\Lambda}$ are given by $\bar{\lambda}_{slt,s'l't'}\pow{is}$ which denotes the covariance between the $(s-1)dk+(l-1)d+t$-th and the $(s'-1)dk+(l'-1)d+t'$-th elements of $\bar \xi$.
    Value $\bar{\lambda}_{slt,s'l't'}$ has the expression
    \begin{align*}
    \bar{\lambda}_{slt,s'l't'} =  \frac{\indicator[(s,l)=(s',l')]\lp \indicator[t=t'] M_{s,t}\pow l-M_{s,t}\pow l M_{s',t'}\pow {l'}\rp}{\sqrt{p_s\pow l p_{s'}\pow{l'}}}.\numberthis\label{eq:correlation-is} 
    \end{align*}
\end{corollary}
While the statement is more general than the one in Theorem \ref{thm:eta-clt}, due to generous scaling, we require Assumption \ref{ass:semi-ergodic}.

\subsection{Non-Existence of CLT}
The validity of the preceding CLTs relies critically on the assumption that every state–action pair is visited infinitely often with a sufficiently regular frequency. 
When this condition fails—such as under transient or weakly exploratory control policies—the empirical transition estimates may not satisfy any CLT irrespective of scalings. 
This subsection formalizes these failure cases and delineates the boundary between recurrent regimes admitting asymptotic normality and those where no normalization yields a tight limiting distribution.

The following proposition entails occasions when there exists no CLT.


\begin{proposition}\label{thm:no-clt}
If there exists a state-action pair $(s,l)\in \mathcal S \times \mathcal A$ such that $\mathbb P(a_i=l|X_i=s)$ is independent of $d$ for all $i \geq 0$, and
\[
\sum_{i=0}^{\infty} \mathbb P(a_i=l|X_i=s) < \infty,
\]
then
    \[
    \lim_{n\to\infty} \sum_{i=1}^{n} \prob(X_i=s,a_i=l) < \infty,
    \]
and there exists a CMC with transition kernel $M$ such that for any possible sequence of estimators  
\(\{\hat M_{s,t}\pow l\}_{n\geq 1}\) and for any positive sequence \(\{b_n\}\) with \(b_n\rightarrow\infty\),
\begin{equation}\label{eq:tight}
    \left\{b_n\;
    \sup_{s,l,t}\bigl|
        \hat M_{s,t}\pow l(n) - M_{s,t}\pow l
    \bigr|\right\}_{n \geq 1}
\end{equation}
is not tight i.e., for every $\epsilon > 0$, 
\[
\limsup_{n \to \infty} \mathbb P \lp b_n\;
    \sup_{s,l,t}\bigl|
        \hat M_{s,t}\pow l(n) - M_{s,t}\pow l
    \bigr| > \epsilon \rp > 0.
\]
\end{proposition}
The proof of this proposition can be found in Appendix~\ref{sec:proof-no-clt}.

We remark that in contrast to Proposition \ref{prop:count-lower-bound}, it can be verified using the Borel-Cantelli lemma that $\lim_{n\to\infty} \sum_{i=1}^{n} \prob(X_i=s,a_i=l) < \infty$ implies \[\sum_{i=1}^{\infty} \prob\lp N_s \pow l(\infty) \geq i-1\rp/T_i < \infty\]
    where $N_s \pow l(\infty) := \lim_{n\to \infty} N_s \pow l(n)$.

Unlike in the uncontrolled setting, null-recurrence properties of CMCs are entwined with non-stationarity of the control sequence. 
Some controls may ensure persistent exploration of the entire state-action space, while others may confine the trajectory to a restricted subset. Consequently, the sharp positive-recurrent/null-recurrent/transient classification in Markov chains is replaced by a policy-dependent continuum in CMCs. 
This leads to a ``grey area" in defining null-recurrence, where the trajectory may drift widely under non-stationary controls while still revisiting a fixed state–action pair every so often.
We classify null-recurrence as the razor's edge between the regimes of Propositions \ref{prop:count-lower-bound} and \ref{thm:no-clt}. 
The statistical properties of such a controlled Markov chain are difficult to characterize by the current analysis. 
However, since exploration problems are not necessarily recurrent \citep{mutti_importance_2022}, this remains an important open direction for future studies.

\subsection{Goodness-of-Fit Test for the Transition Probabilities}
Beyond asymptotic characterization, the properly scaled CLT insinuates a method for direct statistical inference on transition probabilities through a goodness-of-fit (G.O.F.) test that assesses whether an estimated transition kernel is consistent with a hypothesized model. 
The test statistic leverages the covariance structure $\Lambda$ from Theorem \ref{thm:eta-clt}, leading to a chi-square-type limit under the null hypothesis and providing a practical diagnostic for model parameters.
One natural example is testing whether the observed data arise from a fully controlled MDP---where the next state depends on both state and action---or from a contextual bandit model, in which transitions are independent of the previous state.
This can be thought of as a CMC counterpart of the Scheffé's multiple comparison test \citep{scheffe1953method}. 

Formally, let $\{M \pow l_{s,t}\}_{(s,l,t)\in\mathcal S\times\mathcal A\times\mathcal S}$ denote a fixed, hypothesized transition kernel.
We consider the null hypothesis
\[
H_0:\quad \mathbb{P}(X_{i+1}=t \mid X_i=s, a_i=l) = M^{(l)}_{s,t}
\quad \text{for all } (s,l,t)\in\mathcal S\times\mathcal A\times\mathcal S.
\]
Under $H_0$, Theorem~\ref{thm:eta-clt} implies that the scaled estimation errors are asymptotically Gaussian, which in turn yields chi-square test statistics.
The following proposition states the resulting asymptotic distribution.

\begin{proposition}\label{prop:gof}
Given any $(s,l) \in \mathcal{S}\times\mathcal{A}$, under the null hypothesis $H_0$ and Assumptions~\ref{ass:return-time-growth} and ~\ref{ass:eta-mix},
\begin{align*}
    \sum_{t \in \mathcal S} \frac{\lp N_{s,t} \pow l - N_s \pow l M_{s,t} \pow l \rp^2}{N_s \pow l M_{s,t} \pow l}\overset{d}{\rightarrow} \chi^2_{d_{(s,l)} - 1},
\end{align*}
where $\chi^2_{d_{(s,l)} - 1}$ is the chi-square distribution with $d_{(s,l)} - 1$ degrees of freedom, and $d_{(s,l)}=\sum_{t\in\mathcal S}\indicator [M_{s,t} \pow l > 0]$.

Furthermore, the $dk$ chi-square statistics $\{\chi^2_{d_{(s,l)} - 1}\}_{(s,l) \in \mathcal{S} \times\mathcal{A}}$ are asymptotically independent, and 
\begin{align}
\label{eq:GOF}
    \sum_{(s, l) \in \mathcal S \times \mathcal A, t \in \mathcal S} \frac{\lp N_{s,t} \pow l - N_s \pow l M_{s,t} \pow l \rp^2}{N_s \pow l M_{s,t} \pow l} \overset{d}{\rightarrow} \chi^2_{\sum_{(s,l) \in \mathcal S \times \mathcal A} d_{(s,l)} - dk}.
\end{align}
\end{proposition}

\cref{eq:GOF} hence provides a G.O.F. measure for assessing the validity of assumed transition probabilities $\{M_{s,t} \pow l\}$.
The proof of this proposition can be found in Appendix~\ref{sec:proof-GOF}.

\section{CLTs for the Value, Q-, and Advantage Functions, and the Optimal Policy Value}
\label{sec:RL}

We now demonstrate how Theorem~\ref{thm:eta-clt} and Corollary~\ref{cor:improper-clt} extend naturally to downstream tasks in offline RL, such as \emph{offline policy evaluation} (OPE) and \emph{offline policy recovery} (OPR).  
In OPE, one seeks to estimate the value of a fixed target policy from logged trajectories generated under an unknown behavior policy, whereas in OPR, one aims to identify the optimal policy that maximizes the expected return given an estimated transition model.  
Both problems rely on accurate estimation of the value, Q-, and advantage functions, for which we now establish joint asymptotic normality results.  

Let $\Delta(\mathcal{A})$ denote the probability simplex over the action space $\mathcal A$, and let $\pi: \mathcal{S} \to \Delta(\mathcal{A})$ be a stationary target policy.
In offline RL, the (known) target policy is usually different from the (unknown) logging policy under which we obtain the logged data $\indexeddata$.

For each state $s$, denote $\pi_s = \lb \pi(s,1), \ldots, \pi(s,k) \rb$ and define $\Pi = \operatorname{diag}(\pi_1, \ldots, \pi_d)$ as a $d\times dk$ block-diagonal matrix.
Let $\tilde g := (\tilde g(x,a):(x,a)\in\mathcal{S}\times\mathcal{A})$ denote the per-state-action reward vector, and define the per-state expected reward
\[
g(x) = \sum_{a\in\mathcal{A}} \pi(x,a)\,\tilde g(x,a), \,
g = (g(x):x\in\mathcal{S})\in\mathbb{R}^d.
\]
For a discount factor $0<\alpha<1$, the value function is defined by
\[
V_\Pi = (I-\alpha\Pi \mathbf{M})^{-1} g, 
\]
and its plug-in estimate
$\hat V_\Pi = (I-\alpha\Pi\hat{\mathbf{M}})^{-1} g$
follows from the Bellman equation \citep{bertsekas2011dynamic}.

Similarly, let $\tilde r := (\tilde r(x,a,y):(x,a,y)\in\mathcal{S}\times\mathcal{A}\times\mathcal{S})$ denote the immediate reward on transition from $x$ to $y$ under action $a$, and define
\[
r(x,a) = \sum_{y\in\mathcal{S}} M\pow a_{x,y}\,\tilde r(x,a,y), \,
r=(r(x,a):(x,a)\in\mathcal{S}\times\mathcal{A})\in\mathbb{R}^{dk}.
\]
Then the Q-function satisfies \citep{agarwal2019reinforcement}
\[
Q_\Pi = (I-\alpha \mathbf{M}\Pi)^{-1} r,
\,
\hat Q_\Pi = (I-\alpha \hat{\mathbf{M}}\Pi)^{-1} r.
\]
Finally, the advantage function $A_\Pi = Q_\Pi - K V_\Pi$, with $K = \operatorname{diag}(\mathbf{1}_k,\ldots,\mathbf{1}_k)$, admits plug-in estimate $\hat A_\Pi = \hat Q_\Pi - K \hat V_\Pi$.

For compactness, we define the following matrices corresponding to the three functions:
\begin{align*}
\Sigma_V^\Pi &= \alpha^2\!\left[
 (I-\alpha\Pi\mathbf{M})^{-1}\Pi\otimes V_\Pi^\top
 \right]\!\bar\Lambda
 \!\left[
 \Pi^\top(I-\alpha\mathbf{M}^\top\Pi^\top)^{-1}\otimes V_\Pi
 \right],\\[0.4em]
\Sigma_Q^\Pi &= \alpha^2\!\left[
 (I-\alpha\mathbf{M}\Pi)^{-1}\!\otimes Q_\Pi^\top\Pi^\top
 \right]\!\bar\Lambda
 \!\left[
 (I-\alpha\Pi^\top\mathbf{M}^\top)^{-1}\!\otimes \Pi Q_\Pi
 \right],\\[0.4em]
\Sigma_A^\Pi &= \alpha^2\Big[
  (I-\alpha\mathbf{M}\Pi)^{-1}\!\otimes Q_\Pi^\top\Pi^\top
  - K \lp (I-\alpha\Pi\mathbf{M})^{-1}\Pi\otimes V_\Pi^\top \rp
 \Big]\!\bar\Lambda\\[-0.3em]
&\hspace{6.3em}\cdot
 \Big[
 (I-\alpha\Pi^\top\mathbf{M}^\top)^{-1}\!\otimes \Pi Q_\Pi
 - (\Pi^\top(I-\alpha\mathbf{M}^\top\Pi^\top)^{-1}\otimes V_\Pi)K^\top
 \Big],
\end{align*}
where $\bar\Lambda$ is the improperly scaled covariance matrix in Corollary~\ref{cor:improper-clt}.

The following theorem shows that $\Sigma_V\pow \Pi$, $\Sigma_Q\pow \Pi$ and $\Sigma_A\pow \Pi$ are the asymptotic covariance matrices for the scaled estimation errors $ \sqrt{n}\lp\hat V_\Pi - V_\Pi\rp$, $ \sqrt{n}\lp\hat Q_\Pi - Q_\Pi\rp$, $ \sqrt{n}\lp\hat A_\Pi - A_\Pi\rp$, respectively.
\begin{theorem}[CLTs for $V_\Pi$, $Q_\Pi$, and $A_\Pi$]
\label{thm:VQA}
Let $\indexeddata$ be a sample from a controlled Markov chain under a logging policy. 
Suppose that Assumptions~\ref{ass:return-time-growth}, \ref{ass:eta-mix}, and \ref{ass:semi-ergodic} hold for the logging policy, and let $\pi$ be a stationary target policy.
Then, as $n\to\infty$,
\begin{align*}
&\sqrt{n}\,(\hat V_\Pi - V_\Pi) \overset{d}{\rightarrow} \Ncal(0,\Sigma_V^\Pi), \\
&\sqrt{n}\,(\hat Q_\Pi - Q_\Pi)
\overset{d}{\rightarrow} \Ncal(0,\Sigma_Q^\Pi), \\
&\sqrt{n}\,(\hat A_\Pi - A_\Pi) \overset{d}{\rightarrow} \Ncal(0,\Sigma_A^\Pi),
\end{align*}
and each estimator is consistent:
$\hat V_\Pi\overset{p}{\to}V_\Pi$, 
$\hat Q_\Pi\overset{p}{\to}Q_\Pi$, and
$\hat A_\Pi\overset{p}{\to}A_\Pi$.
\end{theorem}
The proof of this theorem can be found in Appendix~\ref{prf:VQA}. Observe that Theorem \ref{thm:VQA} derives the asymptotic distribution for the estimated $V/Q/A$ functions. In \Cref{thm:optpol}, we produce the asymptotic distribution of the optimal value function. Similarly, we derive the CLT for the $Q$-function.
\begin{corollary}[Properly scaled CLT for $Q_\Pi$]
\label{cor:Q-proper}
Suppose that Assumptions~\ref{ass:return-time-growth}, \ref{ass:eta-mix}, and \ref{ass:semi-ergodic} hold for the logging policy, and let $\pi$ be a stationary target policy, then the properly scaled Q-estimation error satisfies
\[
(\hat Q_\Pi - Q_\Pi) \odot  \Vec{\mathbf{N}^\top} \overset{d}{\rightarrow} 
\Ncal \lp 0, \Lambda_Q^\Pi \rp,
\]
where the $(sl,s'l')$-th element of $\Lambda_Q^\Pi$ equals 
$\lambda_{sl,s'l'}^{Q_\Pi}=\sqrt{p_s \pow l\,p_{s'} \pow{l'}}\,\sigma_{sl,s'l'}^{Q_\Pi}$, 
with $\sigma_{sl,s'l'}^{Q_\Pi}$ denoting the $(sl,s'l')$-th element of $\Sigma_Q^\Pi$.
\end{corollary}
The proof of this corollary can be found in Appendix~\ref{sec:prf-Q-thm}.

Although the Q-estimation error is scaled by the visitation counts, this ``correct'' scaling does
not remove dependence on the ergodic occupation measure in the asymptotic covariance.
The reason is structural rather than technical.
The count-based scaling normalizes the local estimation noise of each transition probability,
rendering the errors asymptotically multinomial.
However, the Q-function is a global functional of the transition kernel: estimation errors
from different state–action pairs are propagated and aggregated through repeated
applications of the Bellman operator.
As a result, the asymptotic variance depends on the long-run visitation frequency of
state-action pairs under the logging policy.

The optimal policy can be found, for instance, by policy iteration \citep[Chapter 1]{bertsekas2011dynamic} or policy gradient \citep[Chapter 13]{sutton2018reinforcement}. 
For any policy $\pi$ let $\Pi$ be its block diagonal matrix and let $\Pi_{opt}$ and $\hat \Pi_{opt}$ be the block-diagonal matrices associated with the optimal policies $\pi_{opt}$ and $\hat \pi_{opt}$, each maximizing the expected reward under the true and estimated transition matrices $\mathbf{M}$ and $\mathbf{\hat M}$, respectively. 
The following theorem characterizes the asymptotic behavior of plug-in estimator of the value function  for the estimated optimal policy.
\begin{theorem}\label{thm:optpol}
    Let us assume that for any $\eps>0$, and policy $\pi'$ such that $\lV \Pi_{opt}-\Pi'\rV_{\infty}>\eps$, the value functions corresponding to $\pi_{opt}$ and $\pi'$ satisfy  
    $\inf_{s\in \mathcal{S}}\lc V_{\Pi_{opt}}(s)-V_{\Pi'}(s)\rc>0$. 
    Then $\hat \Pi_{opt}\overset{p}{\rightarrow}\Pi_{opt}$ and,
    \begin{align*}
         \sqrt{n}\lp\hat V_{\hat \Pi_{opt}}-V_{\Pi_{opt}}\rp \overset{d}{\rightarrow} \Ncal\lp0,\Sigma_V\pow{\Pi_{opt}}\rp.
    \end{align*}
\end{theorem}
The proof of this theorem can be found in Appendix~\ref{prf:optpol}.

\paragraph{Implications for OPE and OPR.}
Taken together, Theorems~\ref{thm:VQA} and \ref{thm:optpol} yield an asymptotic
framework for model-based offline reinforcement learning.
We now formalize how the established CLTs can be used to construct asymptotically valid confidence intervals for OPE and OPR.

Without loss of generality, we consider inference for the value function.
Given logged data $\{(X_i, a_i)\}_{i=0}^{n-1}$ generated under a possibly history-dependent logging policy, and a fixed stationary target policy $\pi$, the goal of OPE is to infer the value function $V_\Pi$.
By Theorem~\ref{thm:VQA}, the scaled estimation error of $V_\Pi$ is jointly asymptotically normal:
\[
\sqrt{n}\,(\hat V_\Pi - V_\Pi) \overset{d}{\rightarrow} \Ncal(0,\Sigma_V^\Pi).
\]
Next, let $\hat \Sigma_V^\Pi$ denote the plug-in estimator of $\Sigma_V^\Pi$. Then
\[
\hat \Sigma_V^\Pi = \alpha^2\!\left[
 (I-\alpha\Pi\hat{\mathbf{M}})^{-1}\Pi\otimes \hat V_\Pi^\top
 \right]\! \hat{\bar{\Lambda}}
 \!\left[
 \Pi^\top(I-\alpha \hat{\mathbf{M}}^\top\Pi^\top)^{-1}\otimes \hat V_\Pi
 \right],
\]
where the elements of $\hat{\bar{\Lambda}}$ are given by
\begin{align*}
    \hat{\bar{\lambda}}_{slt,s'l't'} =  \frac{\indicator[(s,l)=(s',l')]\lp \indicator[t=t'] \hat M_{s,t}\pow l- \hat M_{s,t}\pow l \hat M_{s',t'}\pow {l'}\rp}{\sqrt{\hat{p}_s\pow l \hat{p}_{s'}\pow{l'}}}, \ \hat{p}_s\pow l = \frac{N_s \pow l}{n}.
\end{align*}

For a target level $1-\delta$, the confidence set
\begin{align*}
\operatorname{CS}^{\delta}_{V_\Pi}
=\left\{\, v\in\mathbb{R}^{d}:\;
n\,(\hat V_\Pi - v)^{\top}\big(\widehat{\Sigma}^{(\Pi)}_V\big)^{-1}(\hat V_\Pi - v)
\le \chi^2_{1-\delta}(d)\,\right\}
\end{align*}
is an asymptotically valid $100(1-\delta)\%$ confidence set for $V_\Pi$, where $\chi^{2}_{1-\delta}(d)$ is the $(1-\delta)$-quantile of the chi-squared distribution with $d = |S|$ degrees of freedom.

Let $\sigma^{2}_{s}$ denote the marginal variance of the $s$-th scaled estimation error $\sqrt{n}\lp \hat V_\Pi(s) - V_\Pi(s) \rp$, the $(s,s)$-th entry of $\Sigma_V^\Pi$ and let $\hat \sigma^{2}_{s}$ denote its plug-in estimator, the $(s,s)$-th entry of $\hat \Sigma_V^\Pi$.
Then an approximate two-sided $100(1-\delta)\%$ confidence interval for $V_\Pi(s)$ is
\begin{align*}
\hat V_\Pi(s)
\ \pm\
z_{1-\delta/2}\,
\sqrt{\frac{\hat{\sigma}^{2}_{s}}{n}},
\ s \in \mathcal{S},
\end{align*}
where $z_{1-\delta/2}$ is the $(1-\delta/2)$-quantile of the standard normal distribution.
The derivation of the confidence sets and intervals for $Q_\Pi$ and $A_\Pi$ are analogous.

Both the ellipsoidal confidence sets and the marginal confidence intervals above are first-order asymptotically valid. 
A systematic investigation of their higher-order accuracy and finite-sample performance—via Edgeworth expansions and potential improvements using model-based bootstrap—makes an interesting problem for future work.

\section{Conclusion}\label{sec:conclusion}

We develop the first asymptotic theory for \emph{non-parametric} estimation in finite CMCs. 
Our main result (Theorem~\ref{thm:eta-clt}) establishes a central limit theorem for the classical count–based estimator of state–action transition probabilities.
This non-trivially extends prior results to accommodate non-Markovian and non-stationary data, which renders classical regenerative techniques inapplicable.
Our proof introduces a new auxiliary sampling scheme that decouples temporal dependence while preserving the finite-state controlled dynamics, thereby extending the sampling construction of \citet{billingsley1961statistical} to the controlled and time-inhomogeneous setting.

Building on this foundation, we show that plug-in estimators of the value, Q-, and advantage functions under arbitrary stationary target policies—as well as the value of the estimated optimal policy—also satisfy CLTs.
All results hold under general, possibly history-dependent logging policies, providing a unified large-sample framework for model-based offline policy evaluation (OPE) and optimal policy recovery (OPR) in reinforcement learning.

Several directions for future work are possible. 
Analyzing the higher-order accuracy and finite-sample performance of the confidence sets and intervals for OPE and OPR derived from our asymptotic results would provide a deeper understanding of their practical reliability in offline reinforcement learning.
Refining the testability criteria to close the gap between the current sufficiency results and necessary impossibility statements would delineate precisely when asymptotic inference is attainable. 
Finally, extending the theory to infinite \citep{kontorovich_concentration_2008} or continuous state–action spaces, as suggested by the adaptive density estimation work in Markov settings \citep{sart_estimation_2014}, and to continuous-time controlled processes, would markedly broaden the scope of rigorous statistical inference of controlled Markov chains and its applications in reinforcement learning.

\newpage

\appendix
\section{Examples of Controlled Markov Chains}
\subsection{Inhomogeneous Markov Chains}
\label{sec:inhomo-chain}
We consider an inhomogeneous Markov chain as the first example. 
A controlled Markov chain is said to be an inhomogeneous Markov chain if there exists a sequence of actions $l_0,l_1,\ldots \in \mathcal A$ such that for any time $i$, state $s$, sample history $\history_0^{i-1}$, we have
\[
\prob(a_i=l_i|\History_0^{i-1}=\history_0^{i-1},X_i=s)=1.
\]
We make the following assumption on the logging policy.
\begin{assumption}
\label{ass:inhomo_MC}\
\begin{enumerate}
    \item Let $T_i = O\lp i^\alpha\rp$, $\alpha \in (0,1)$ be a sequence of known integers. 
Then, we assume that
\begin{align*}
    \sum_{p=j}^{T_i+j} \indicator[a_p=l]\geq1,\numberthis\label{eq:inhomogenous-eq2}
 \end{align*}
for all $\tau_{s,l} \pow{i-1} + 1\leq j\leq \tau_{s,l} \pow i-T_i$ and every $l\in \mathcal A$.
    \item There exists $M_{min}$ and $M_{max}$ such that for all $s,t\in \mathcal{S}$ and $l\in\mathcal{A}$,
    \begin{align}
    0<M_{min}\leq M_{s,t}\pow l\leq M_{max}<1.   \label{eq:I2}     
    \end{align}
\end{enumerate}    
\end{assumption}
\begin{proposition}~\label{prop:inhomogenous-mc}
Let $\indexeddata$ be a sample from an inhomogeneous Markov chain satisfying Assumption~\ref{ass:inhomo_MC}. Then Theorem~\ref{thm:eta-clt} holds with
\[
C=0, \ C_{\theta}= e/(e-1).
\]
\end{proposition}

\begin{proof}
We begin our proof by verifying Assumption~\ref{ass:return-time-growth}.
For any $k>i$, (\ref{eq:I2}) implies that
    \begin{align*}
    \prob\lp \tau_{s,l}\pow{i}>k | \Fcal_{\sum_{p=1}^{i-1} \tau_{s,l}\pow{p}} \rp  & =\prob\lp \lc X_j\neq s\bigcup a_j\neq l \rc \forall j\in\{i+1,\dots,k+i\}|X_i=s,a_i=l \rp\\
    & \leq \lp1-M_{min}\rp^{\lfloor\frac{k}{T_i}\rfloor}\\
    & \leq \lp1-M_{min}\rp^{\frac{k}{T_i}-1}.
\end{align*}
Thus, 
\begin{align*}
    \expec[\tau_{s,l} \pow i | \Fcal_{\sum_{p=1}^{i-1}\tau_{s,l}\pow{p}} ]=\sum_{k=1}^\infty \prob\lp\tau_{s,l}^{(i)}>k |\Fcal_{\sum_{p=1}^{i-1} \tau_{s,l}\pow{p}}\rp&\leq \frac{\lp1-M_{min}\rp^{\frac{1}{T_i}-1}}{\lp 1-\lp1-M_{min}\rp^{\frac{1}{T_i}}\rp}. \numberthis\label{eq:inhomogenous-eq3}
\end{align*}
As $T_i = O(i^\alpha)$, there exists a constant $c>0$ such that $T_i \leq c i^\alpha$ for sufficiently large $i$.

Let $q := 1 - M_{min}$, $\widetilde T_i := \nicefrac{\log q}{\log\lp \nicefrac{c i^\alpha}{c i^\alpha - \log q} \rp}$.
Then 
\begin{align*}
    \widetilde T_i & = \frac{\log q}{\log\lp \frac{c i^\alpha}{c i^\alpha - \log q} \rp} \\
    &= - \frac{\log q}{\log \lp 1 - \frac{\log q}{c i^\alpha}\rp} \numberthis \label{eq:inhomogeneous-eq4}
\end{align*}
By the Taylor expansion of $\log(1+x)$, $\log(1+x) = x - o(x)$ as $x \to 0^+$. 
Then for sufficiently large $i$, the right-hand side of (\ref{eq:inhomogeneous-eq4}) becomes
\begin{align*}
     \widetilde T_i   &= - \frac{\log q}{-\frac{\log q}{c i^\alpha} - o(1/i^\alpha)} \\
    &= \frac{-c i^\alpha \log q}{-\log q - o(1)} \\
    & \geq \frac{-c i^\alpha \log q}{-\log q} \\
    &= c i^\alpha = T_i.
\end{align*}
Thus, $- \nicefrac{1}{\widetilde T_i} \geq - \nicefrac{1}{T_i}$, and $q^{-\nicefrac{1}{\widetilde T_i}} \leq q^{-\nicefrac{1}{T_i}}$.
Therefore,
\begin{align*}
    \frac{q^{\frac{1}{T_i}-1}}{1- q^{\frac{1}{T_i}}} &= \frac{q^{-1}}{ q^{-\frac{1}{T_i}} - 1} \leq \frac{ q^{-1}}{q^{-\frac{1}{\widetilde T_i}} - 1} = \frac{q^{\frac{1}{\widetilde T_i}-1}}{1-q^{\frac{1}{\widetilde T_i}}}.
\end{align*}
Substituting this value into the right hand side of (\ref{eq:inhomogenous-eq3}), we get
\begin{align*}
    \expec[\tau_{s,l}^{(i)}] & \leq \frac{q^{\frac{1}{\widetilde T_i}-1}}{1-q^{\frac{1}{\widetilde T_i}}} \\
     &= \frac{q^{\frac{\log\lp \frac{c i^\alpha}{c i^\alpha - \log q} \rp}{\log q}-1}}{1-q^{\frac{\log\lp \frac{c i^\alpha}{c i^\alpha - \log q} \rp}{\log q}}}\\
    &=\frac{\frac{c i^\alpha}{c i^\alpha - \log q}}{q-\frac{q c i^\alpha}{c i^\alpha - \log q}}\\
    &= \frac{c}{-q \log q} i^\alpha.
\end{align*}
Hence, Assumption~\ref{ass:return-time-growth} holds.

Now we verify Assumptions~\ref{assume:control-mixing} and \ref{assume:chain-mix}.
As the controls are deterministic, the total variation distance 
\begin{align*}
\gamma_{p,j,i} & = 
    \sup_{s_p, \history_{i+j}^{p-1}, \history_0^i} 
    \Big\lVert 
        \probl\lp 
            a_p \Big| X_p = s_p, \History_{i+j}^{p-1} = \history_{i+j}^{p-1}, \History_0^i = \history_{0}^i
        \rp - 
        \probl\lp  
            a_p \Big| X_p = s_p, \History_{i+j}^{p-1} = \history_{i+j}^{p-1}
        \rp
    \Big\rVert_{TV} \\
    & \equiv 0.
\end{align*}
Therefore, Assumption~\ref{assume:control-mixing} holds for all inhomogeneous Markov chains with $C=0$. Observe from the definition of $\bar \theta_{i,j}$ in (\ref{eq:def-theta}) that,
\begin{align*}
\bar \theta_{i,j} =\sup_{l,l',s_1,s_2} \nrm{\probl\lp X_j|X_i=s_1,a_i=l \rp-\probl\lp X_j|X_i=s_2,a_i=l'\rp}_{TV}.
\end{align*}
From the definition of an inhomogenous Markov chains it follows that 
\begin{align*}
    \sup_{s_1, s_2} \bigg\| 
    & \probl \lp X_j \mid X_i = s_1, a_i = l_1\rp 
    - \probl \lp X_j \mid X_i = s_2, a_i = l_2\rp \bigg\|_{TV} \\
    = \sup_{s_1, s_2} \bigg\| 
    & \probl \lp X_j \mid X_i = s_1, (a_{j-1}, \dots, a_i) = (l_{j-1}, \dots, l_i)\rp \\
    - & \probl \lp X_j \mid X_i = s_2, (a_{j-1}, \dots, a_i) = (l_{j-1}, \dots, l_i)\rp 
    \bigg\|_{TV}
\end{align*}
where the last line follows due the controls being deterministic. 

Since all transition matrices are positive, an application of \citet[Theorem 1]{wolfowitz1963products} implies that there exists an integer $C$ for which, 
\[
\bar \theta_{i,j}\leq e^{-C(j-i)}.
\]
Since $e^{-C(j-i)}<e^{-(j-i)}$ for any integer $C$, it consequently implies that,
\[
\sup_{i\geq 1}\sum_{j >i} \bar \theta_{i,j}\leq \frac{e^{-1}}{1-e^{-1}}\leq \frac{1}{1-e^{-1}} = \frac{e}{e-1}. 
\]
Therefore, Assumption~\ref{assume:chain-mix} holds with $C_\theta = \nicefrac{e}{e-1}$. 
Then, by \citet[Lemma 3]{banerjee2025off}, Assumption~\ref{ass:eta-mix} holds. 
As both Assumptions~\ref{ass:return-time-growth} and \ref{ass:eta-mix} hold, Theorem~\ref{thm:eta-clt} holds.
This completes the proof.
\end{proof}
\subsection{Controlled Markov Chains with Non-Stationary Markov Controls}
\label{sec:non-stat-chain}
We consider a controlled Markov chain with non-stationary control process as the second example.
A controlled Markov chain is said to have non-stationary Markov controls if for any time $i$, state $s$, action $l$, and sample history $\history_0^{i-1}$,
\[
\probl\lp a_i=l|X_i=s,\History_0^{i-1}=\history_0^{i-1}\rp=\probl\lp a_i=l|X_i =s\rp.
\]
We observe that an action can depend on the time $i$. 
Let $P_{s,l}^{(i)} := \prob(a_i=l|X_i=s)$.
We can write the transition probability of the state-action pair as
\begin{align*}
    &\prob\lp X_i=t,a_i=l'|X_{i-1}=s,a_{i-1}=l\rp\\
     = \, & \prob\lp X_i=t|X_{i-1}=s,a_{i-1}=l\rp\prob(a_i=l'|X_i=t)
     \\
     = \, & M_{s,t}^{(l)}\cdot P_{t,l'}^{(i)}.
\end{align*}
It is straightforward to see that the state-action pair is a time inhomogeneous Markov chain with transition probabilities given by $M_{s,t}^{(l)} P_{s,l'}^{(i)}$. 
The goal is to perform statistical inference on the transition probabilities $\prob(X_i=t|X_{i-1}=s,a_{i-1}=l')$. 
We proceed by making assumptions on the return times of the actions.
\begin{definition}
We define $\tau_{s,l}\pow{i,\star,j}$ to be the time between the $(j-1)$-th and $j$-th visit to action $l$ after visiting state-action pair $s,l$ for the $i$-th time. 
For ease of notation, let $ \sum_{k=1}^i\tau_{s,l}\pow{k}+\sum_{k=1}^{j-1}\tau^{(i,\star,k)}_{s,l}=\tau_\star$. We can then define $\tau_{s,l}\pow{i,\star,j}$ recursively as
$\tau_{s,l}\pow{i,\star,j}:= \min\{n > 0: a_{\tau_\star+n}=l,a_{j}\neq l, \ \forall \ \tau_\star<j<\tau_\star+n\}.
$
\end{definition}
Next we make some simplifying assumptions on $\tau_{s,l}\pow{i,\star,j}$ and $M\pow{l}$.
\begin{assumption}~\label{ass:markov}
\begin{enumerate}
    \item For all state-action pairs $(s,l) \in \mathcal{S}\times\mathcal{A}$, there exists $\alpha \in (0, 1)$ such that
    \[ \expec[\tau_{s,l}\pow{i,\star,j}|\Fcal_{\sum_{p=1}^{i-1} \tau_{s,l}\pow{p}+\sum_{p=1}^{j-1}\tau_{s,l}\pow{i,\star,p}}] \leq T_i^\star  \text{ almost surely, with }T_i^\star = O(i^\alpha).
    \]
    \item There exists $M_{min}$ and $M_{max}$ such that for all $s,t\in \mathcal{S}$ and $l\in\mathcal{A}$,
    \begin{align}
    0<M_{min}\leq M_{s,t}\pow l\leq M_{max}<1.   \label{eq:markov-eq1}     
    \end{align}
\end{enumerate}
\end{assumption}
The next lemma proves that under this assumption $\{(X_i,a_i)\}$ satisfies Assumption \ref{ass:return-time-growth}.
\begin{lemma}\label{lemma:return-time-markov}
Under the conditions of Assumption \ref{ass:markov}, for all $(i,s,l)\in\naturalset\times\mathcal{S}\times\mathcal{A}$, it holds almost surely that

    \begin{align*}
    \expec[\tau_{s,l}^{(i)}|\Fcal_{\sum_{p=1}^{i-1} \tau_{s,l}\pow{p}}] \leq \frac{T_i^\star M_{max}}{(1-\max\{M_{max},1-M_{min}\})^2}.\numberthis\label{eq:markov-eq3}         
    \end{align*}

\end{lemma}
\begin{proof}
    We begin by observing that
\begin{align*}
  \tau_{s,l}\pow{i+1}
  =\begin{cases}
      \tau_{s,l}\pow{i,\star,1}
      & \text{if } 
      X_{\sum_{p=1}^i \tau_{s,l}\pow{p}+ \tau_{s,l}\pow{i,\star,1}} = s,
      \\[6pt]
      \displaystyle \sum_{k=1}^j \tau_{s,l}\pow{i,\star,k}
      & \text{if }
        \begin{aligned}[t]
          &X_{\sum_{p=1}^i \tau_{s,l}\pow{p}+ \sum_{k=1}^{m}
             \tau_{s,l}\pow{i,\star,k}} \neq s,
          &&\forall\, 1 \le m \le j
          \\[2pt]
          &\text{and }\;
          X_{\sum_{p=1}^i \tau_{s,l}\pow{p}
             + \sum_{k=1}^j \tau_{s,l}\pow{i,\star,k}} = s.
        \end{aligned}
    \end{cases}
\end{align*}
Therefore, 
\begin{align*}
    \tau_{s,l}\pow{i+1}
    &= \tau_{s,l}\pow{i,\star,1}\,
       \indicator\lb
         X_{\sum_{p=1}^i \tau_{s,l}\pow{p}+ \tau_{s,l}\pow{i,\star,1}} = s
       \rb \\[4pt]
    &\quad
      + \sum_{j=1}^\infty \sum_{k=1}^j \tau_{s,l}\pow{i,\star,k}\,
        \indicator\lb
          \begin{aligned}[c]
            &X_{\sum_{p=1}^i \tau_{s,l}\pow{p}
                + \sum_{k=1}^{m}\tau_{s,l}\pow{i,\star,k}} \neq s,
                &&\forall m \in \{1,\ldots,j\},\\
            &X_{\sum_{p=1}^i \tau_{s,l}\pow{p}
                + \sum_{k=1}^{j}\tau_{s,l}\pow{i,\star,k}} = s
          \end{aligned}
        \rb,\\
\intertext{which in turn implies that}
    &\expec[\tau_{s,l}\pow{i+1}\mid \Fcal_{\sum_{p=1}^{i-1}\tau_{s,l}\pow{p}}] \\[2pt]
    &=
      \expec\!\Bigl[
        \tau_{s,l}\pow{i,\star,1}\,
        \indicator\lb
          X_{\sum_{p=1}^i \tau_{s,l}\pow{p}+ \tau_{s,l}\pow{i,\star,1}} = s
        \rb
        \,\Bigm|\,
        \Fcal_{\sum_{p=1}^{i-1}\tau_{s,l}\pow{p}}
      \Bigr] \\[4pt]
    &\quad
      + \sum_{j=1}^\infty \sum_{k=1}^j
        \expec\!\Bigl[
          \tau_{s,l}\pow{i,\star,k}\,
          \indicator\lb
            \begin{aligned}[c]
              &X_{\sum_{p=1}^i \tau_{s,l}\pow{p}
                  + \sum_{k=1}^{m}\tau_{s,l}\pow{i,\star,k}} \neq s,
                  &&\forall m \in \{1,\ldots,j\},\\
              &X_{\sum_{p=1}^i \tau_{s,l}\pow{p}
                  + \sum_{k=1}^j \tau_{s,l}\pow{i,\star,k}} = s
            \end{aligned}
          \rb
          \,\Bigm|\,
          \Fcal_{\sum_{p=1}^{i-1}\tau_{s,l}\pow{p}}
        \Bigr] \\[4pt]
    &= \text{Term 1} + \text{Term 2}.
    \numberthis \label{eq:ret-t-markeq4}
\end{align*}

We now analyze the terms one by one. 

\noindent\textbf{Term 1: } The law of conditional expectation implies that
\begin{align*}
    & \expec\lb\tau_{s,l}\pow{i,\star,1}\indicator\lb X_{\sum_{p=1}^i \tau_{s,l}\pow{p}+ \tau_{s,l}\pow{i,\star,1}}=s\rb|\Fcal_{\sum_{p=1}^{i-1} \tau_{s,l}\pow{p}}\rb \\
     = \ & \expec\lb \expec\lb\tau_{s,l}\pow{i,\star,1}\indicator\lb X_{\sum_{p=1}^i \tau_{s,l}\pow{p}+ \tau_{s,l}\pow{i,\star,1}}=s\rb\gn\tau_{s,l}\pow{i,\star,1} \rb|\Fcal_{\sum_{p=1}^{i-1} \tau_{s,l}\pow{p}} \rb\\
    = \ & \expec\lb\tau_{s,l}\pow{i,\star,1}\prob\lp X_{\sum_{p=1}^i \tau_{s,l}\pow{p}+ \tau_{s,l}\pow{i,\star,1}}=s| \tau_{s,l}\pow{i,\star,1} \rp|\Fcal_{\sum_{p=1}^{i-1} \tau_{s,l}\pow{p}}\rb. \numberthis\label{eq:ret-t-markeq2}
\end{align*}
Recall from (\ref{eq:markov-eq1}) that 
\[
 \max_{s,t,l} M_{s,t}^{(l)}=M_{max}, \text{ and } \min_{s,t,l} M_{s,t}^{(l)}=M_{min}
\]
for two numbers $0<M_{min},M_{max}<1$. 
Then for any time $p$, state $s$, and history $\history_0^{p-1}$,
\begin{align*}
    &M_{min}\leq \prob\lp X_p=s\gn \History_0^{p-1}=\history_0^{p-1}\rp\leq M_{max}, \, \, \text{and } \numberthis \label{eq:E1}\\
    & \prob\lp X_p\neq s\gn \History_0^{p-1}=\history_0^{p-1}\rp\leq M_{opt} := \max\lc  M_{max},1-M_{min}\rc. \numberthis \label{eq:E2}
\end{align*}
It follows from (\ref{eq:E1}) that $\prob\lp X_{\sum_{p=1}^i \tau_{s,l}\pow{p}+ \tau_{s,l}\pow{i,\star,1}}=s| \tau_{s,l}\pow{i,\star,1} \rp\leq M_{max}$. 
Substituting this value in the right hand side of (\ref{eq:ret-t-markeq2}), we get the following upper bound to Term 1:
\[
 \expec\lb\tau_{s,l}\pow{i,\star,1}\indicator\lb X_{\sum_{p=1}^i \tau_{s,l}\pow{p}+ \tau_{s,l}\pow{i,\star,1}}=s\rb|\Fcal_{\sum_{p=1}^{i-1} \tau_{s,l}\pow{p}}\rb\leq \expec\lb \tau_{s,l}\pow{i,\star,1}M_{max}|\Fcal_{\sum_{p=1}^{i-1} \tau_{s,l}\pow{p}}\rb\leq T^\star_i M_{max},
\]
where the last inequality follows from the tower property.

\textbf{Term 2: }We define for ease of notation
\[
\expec^*[\cdot]=\expec[\cdot|\Fcal_{\sum_{p=1}^{i-1} \tau_{s,l}\pow{p}}],
\]
and proceed similarly as before to get
\begin{align*}
    &\expec^*\lb\tau_{s,l}\pow{i,\star,k} \indicator\lb             \begin{aligned}[c]
              &X_{\sum_{p=1}^i \tau_{s,l}\pow{p}
                  + \sum_{k=1}^{m}\tau_{s,l}\pow{i,\star,k}} \neq s,
                  &&\forall m \in \{1,\ldots,j\},\\
              &X_{\sum_{p=1}^i \tau_{s,l}\pow{p}
                  + \sum_{k=1}^j \tau_{s,l}\pow{i,\star,k}} = s
            \end{aligned}\rb\rb\\ 
     = & \ \expec^*\lb\tau_{s,l}\pow{i,\star,k}\prob\lp             \begin{aligned}[c]
              &X_{\sum_{p=1}^i \tau_{s,l}\pow{p}
                  + \sum_{k=1}^{m}\tau_{s,l}\pow{i,\star,k}} \neq s,
                  &&\forall m \in \{1,\ldots,j\},\\
              &X_{\sum_{p=1}^i \tau_{s,l}\pow{p}
                  + \sum_{k=1}^j \tau_{s,l}\pow{i,\star,k}} = s
            \end{aligned}| \tau_{s,l}\pow{i,\star,k}, k=1, \ldots, j \rp\rb.\numberthis\label{eq:ret-t-markeq3}
\end{align*}


The Bayes' theorem gives
\begin{align*}
& \prob\lp             \begin{aligned}[c]
              &X_{\sum_{p=1}^i \tau_{s,l}\pow{p}
                  + \sum_{k=1}^{m}\tau_{s,l}\pow{i,\star,k}} \neq s,
                  &&\forall m \in \{1,\ldots,j\},\\
              &X_{\sum_{p=1}^i \tau_{s,l}\pow{p}
                  + \sum_{k=1}^j \tau_{s,l}\pow{i,\star,k}} = s
            \end{aligned} | \tau_{s,l}\pow{i,\star,k}, k=1, \ldots, j \rp \\
 = \ &  \prob\lp X_{\sum_{p=1}^i \tau_{s,l}\pow{p}+ \sum_{k=1}^j \tau_{s,l}\pow{i,\star,k}}=s|
 \begin{aligned}[c]
    & \tau_{s,l}\pow{i,\star,k},  k=1, \ldots, j, \\
    & X_{\sum_{p=1}^i \tau_ {s,l}\pow{p}+  \sum_{k=1}^m \tau_{s,l}\pow{i,\star,k}}\neq s, \ \forall m \in \{1, \ldots, j\}
 \end{aligned}
 \rp
\\
 & \qquad \qquad \times \prob\lp X_{\sum_{p=1}^i \tau_{s,l}\pow{p}+  \sum_{k=1}^m \tau_{s,l}\pow{i,\star,k}}\neq s \ \forall m \in \{1, \ldots, j\}|\tau_{s,l}\pow{i,\star,k}, k=1, \ldots, j\rp \\
 = \ & \prob\lp X_{\sum_{p=1}^i \tau_{s,l}\pow{p}+ \sum_{k=1}^j \tau_{s,l}\pow{i,\star,k}}=s|\begin{aligned}[c]
    & \tau_{s,l}\pow{i,\star,k},  k=1, \ldots, j, \\
    & X_{\sum_{p=1}^i \tau_ {s,l}\pow{p}+  \sum_{k=1}^m \tau_{s,l}\pow{i,\star,k}}\neq s, \ \forall m \in \{1, \ldots, j\}
 \end{aligned}\rp\\
  & \qquad \qquad \times \prod_{m=1}^{j-1} \prob\lp X_{\sum_{p=1}^i \tau_{s,l}\pow{p}+  \sum_{k=1}^m \tau_{s,l}\pow{i,\star,k}}\neq s |\tau_{s,l}\pow{i,\star,k}, k=1, \ldots, j\rp. \numberthis \label{eq:E3}\\
\end{align*}
Using (\ref{eq:E1}) and (\ref{eq:E2}) to bound the first and the second term from (\ref{eq:E3}) we get
\begin{align*}
    \prob\lp \begin{aligned}
        & X_{\sum_{p=1}^i \tau_{s,l}\pow{p}+  \sum_{k=1}^m \tau_{s,l}\pow{i,\star,k}}\neq s, \ \forall m \in \{1, \ldots, j\} \text{ and } \\ & X_{\sum_{p=1}^i \tau_{s,l}\pow{p}+ \sum_{k=1}^j \tau_{s,l}\pow{i,\star,k}}=s
    \end{aligned} | \tau_{s,l}\pow{i,\star,k}, k=1, \ldots, j \rp & \leq   M_{max}M_{opt}^{j-1}.
\end{align*}
Substituting this value in the right hand side of (\ref{eq:ret-t-markeq3}), we get
\begin{align*}
   & \expec^*\lb\tau_{s,l}\pow{i,\star,k} \indicator\lb X_{\sum_{p=1}^i \tau_{s,l}\pow{p}+  \sum_{k=1}^m \tau_{s,l}\pow{i,\star,k}}\neq s, \ \forall m \in \{1, \ldots, j\} \text{ and }  X_{\sum_{p=1}^i \tau_{s,l}\pow{p}+ \sum_{k=1}^j \tau_{s,l}\pow{i,\star,k}}=s\rb\rb \\
   \leq & \ \expec^*[\tau_{s,l}\pow{i,\star,k}] M_{max}M_{opt}^{j-1} \\
    \leq & \  T_\star^i M_{max}M_{opt}^{j-1},
\end{align*}
where the last inequality follows from the tower property. 
Substituting the obtained upper bounds of Term 1 and Term 2 back into (\ref{eq:ret-t-markeq4}), we get
\begin{align*}
    \expec\lb\tau_{s,l}\pow{i}|\Fcal_{\sum_{p=1}^{i-1} \tau_{s,l}\pow{p}}\rb 
    & \leq \sum_{j=1}^\infty jT_\star^i M_{max} M_{opt}^{j-1}\\
    & = \frac{T_\star^i M_{max}}{1-M_{opt}} \sum_{j=1}^\infty j(1-M_{opt})M_{opt}^{j-1}\\
    & = \frac{T_\star^i M_{max}}{(1-M_{opt})^2}.
\end{align*}
\end{proof}


We can now state our main result about the CLT of a controlled Markov chain with a non-stationary Markov controls. 
\begin{proposition}\label{prop:markov-mdp}
Let $\indexeddata$ be a sample from a controlled Markov chain with non-stationary Markovian controls satisfying Assumption~\ref{ass:markov}, then Theorem~\ref{thm:eta-clt} holds with
\[
C = 0, \ C_\theta=1/(d M_{min}).
\]
\end{proposition}
\begin{proof}
    As Assumption~\ref{ass:markov} holds, Lemma~\ref{lemma:return-time-markov} holds. Then Assumption~\ref{ass:return-time-growth} holds. 
    We just need to verify Assumptions~\ref{assume:control-mixing} and \ref{assume:chain-mix}.

    As the controls are non-stationary Markov, the total variation distance
    \begin{align*}
\gamma_{p,j,i} & = 
    \sup_{s_p, \history_{i+j}^{p-1}, \history_0^i} 
    \Big\lVert 
        \probl\lp 
            a_p \Big| X_p = s_p, \History_{i+j}^{p-1} = \history_{i+j}^{p-1}, \History_0^i = \history_{0}^i
        \rp - 
        \probl\lp  
            a_p \Big| X_p = s_p, \History_{i+j}^{p-1} = \history_{i+j}^{p-1}
        \rp
    \Big\rVert_{TV} \\
    & \equiv 0.
\end{align*}
Therefore, Assumption~\ref{assume:control-mixing} holds for all controlled markov chains with non-stationary controls with $C=0$.

Recall from (\ref{eq:markov-eq1}), $M_{min} = \min_{s,t,l} M_{s,t} \pow l > 0$.
Then, by \citet[Lemma 12]{banerjee2025off}, Assumption~\ref{assume:chain-mix} holds with $C_\theta = 1/(d M_{min})$.
\end{proof}

\section{Statements and Proofs of the Propositions and Lemmas Used in Theorem~\ref{thm:eta-clt}}
\subsection{Statement and Proof of Lemma~\ref{lemma:consistency}}
\label{sec:lemma_consistency}
\begin{lemma}\label{lemma:consistency}
    Let $\indexeddata$ be a sample from a controlled Markov chain. Under Assumptions~\ref{ass:return-time-growth} and \ref{ass:eta-mix},
    \begin{align*}
       \frac{N_s\pow l}{\expec[N_s\pow l]}\xrightarrow{a.s.} 1. 
    \end{align*}
    Further suppose that Assumption~\ref{ass:semi-ergodic} holds, then $N_s \pow l$ is $1$-st order almost surely consistent for $n p_s \pow l$. 
    In other words,
    \begin{align*}
    \frac{N_s\pow l}{n}\xrightarrow{a.s.}  p_s \pow l.
    \end{align*}
\end{lemma}
\begin{proof}
        Using \citet[Lemma 6]{banerjee2025off}, for any $t > 0$, we have:
  \begin{align}
      \prob(|N_s\pow{l}-\expec[N_s\pow{l}]|>t)\leq 2\exp\lp-\frac{t^2}{2n\|\Delta_n\|^2}\rp,
  \end{align}
    where $\|\Delta_n\|=\underset{1\leq i\leq n}{\max} (1+ \Bar{\eta}_{i,i+1}+ \Bar{\eta}_{i,i+2}+\dots \Bar{\eta}_{i,n})$, and $\Bar{\eta}_{i,j}$ are the coefficients.
    Setting $t=\epsilon\expec[N_s\pow l]$ we get
    \begin{align}
      \prob\lp\lv\frac{N_s\pow{l}}{\expec[N_s\pow l]}-1\rv>\eps\rp\leq 2\exp\lp-\frac{\expec[N_s\pow l]^2\eps^2}{2n\|\Delta_n\|^2}\rp.
  \end{align}
  By Assumption~\ref{ass:eta-mix}, $\sup_n \|\Delta_n\|<\infty$. 
  As Assumption~\ref{ass:return-time-growth} holds, it follows from Proposition~\ref{prop:count-lower-bound} that $\expec[N_s\pow l]\geq cn^{\nicefrac{1}{1+\alpha}}$ for sufficiently large $n$, where $c>0$ is a constant, and we get $ \prob\lp\lv{N_s\pow{l}}/{\expec[N_s\pow l]}-1\rv>\eps\rp\leq2 \exp\lp- \nicefrac{c^2 \epsilon^2 n^{(1-\alpha)/(1+\alpha)}}{2 \|\Delta_n\|^2}\rp$.
  
  As $\alpha \in (0,1)$, the series $\sum_{n=1}^\infty 2 \exp\lp- \nicefrac{c^2 \epsilon^2 n^{(1-\alpha)/(1+\alpha)}}{2 \|\Delta_n\|^2}\rp$ converges.
  Observing that $\epsilon$ is arbitrary, using the Borel-Cantelli lemma, we get $\lv{N_s\pow{l}}/{\expec[N_s\pow l]}-1\rv\xrightarrow{\ a.s. \ }0$, hence $\nicefrac{N_s\pow l}{\expec[N_s\pow l]}\xrightarrow{a.s.} 1$. 

Using the Slutsky's theorem, we have $\nicefrac{N_s\pow l}{n }\xrightarrow{a.s.} p_s \pow l$ under Assumption~\ref{ass:semi-ergodic}. This completes the proof.
\end{proof}
\subsection{Statement and Proof of Proposition~\ref{prop:mod-sampling-scheme}}
\label{sec:prop_sampling}
\begin{proposition}\label{prop:mod-sampling-scheme}
   $\lp X_0,a_0,\dots,X_n,a_n\rp$ is identically distributed to $\lp \tilde X_0,\tilde a_0,\dots,\tilde X_n,\tilde a_n\rp$.
\end{proposition}
\begin{proof}
We prove this fact by induction. 
Obviously, $(X_0,a_0)\overset{d}{=}(\tilde X_0,\tilde{a_0})$.
Now, for some $i\geq 1$, let $\lp X_0,a_0,\dots,X_i,a_i\rp$ be identically distributed to $\lp \tilde X_0,\tilde a_0,\dots,\tilde X_i,\tilde a_i\rp$. Then, for $i+1$, we note that,
\begin{align*}
    & \prob\lp \tilde X_{i+1}=s_{i+1},\tilde a_{i+1}=l_{i+1},\dots,\tilde X_0=s_0,\tilde a_0=l_0 \rp\\
    & \ =\prob\lp \tilde a_{i+1}=l_{i+1}| \tilde X_{i+1}=s_{i+1},\dots,\tilde X_0=s_0\rp\\ 
    &\quad \times\prob\lp \tilde X_{i+1}=s_{i+1}|\tilde X_i=s_i,\tilde a_{i}=l_i,\dots,\tilde a_0=l_0,\tilde X_0=s_0\rp\\
    &\quad \times \prob(\tilde X_i=s_i,\tilde a_i=l_i,\dots,\tilde X_0=s_0,\tilde a_0=l_0)\\
    & \ =\prob\lp  \alpha_i\pow{\tilde X_0,\tilde{a_0},\dots,\tilde X_{i+1}}=l_{i+1}| \tilde X_{i+1}=s_{i+1},\dots,\tilde X_0=s_0\rp\\
    &\quad \times \prob\lp X_{\Tilde{X_i},\Tilde{N}_{\Tilde{X_i}}^{(i,\tilde a_i)}+1}\pow{\tilde a_i}=s_{i+1}|\tilde X_i=s_i,\tilde a_{i}=l_i,\dots,\tilde a_0=l_0,\tilde X_0=s_0\rp\\
    &\quad \times \prob(\tilde X_i=s_i,\tilde a_i=l_i,\dots,\tilde X_0=s_0,\tilde a_0=l_0)\\
    & \ =\prob\lp  \alpha_i\pow{s_0,l_0,\dots,s_{i+1}}=l_{i+1}| \tilde X_{i+1}=s_{i+1},\dots,\tilde X_0=s_0\rp\\
    &\quad \times \prob\lp X_{s_i,\Tilde{N}_{s_i}^{(i,l_i)}+1}\pow{l_i}=s_{i+1}|\tilde X_i=s_i,\tilde a_{i}=l_i,\dots,\tilde a_0=l_0,\tilde X_0=s_0\rp\\
    &\quad \times \prob(\tilde X_i=s_i,\tilde a_i=l_i,\dots,\tilde X_0=s_0,\tilde a_0=l_0),
\end{align*}
where the equalities follow by substituting in the corresponding value of each quantity. Observe that under the given conditional, such that $\Tilde{N}_{s_i}^{(i,l_i)}+1$ is some fixed integer $n$. 
Then,
\begin{align*}
    \prob\lp  \alpha_i\pow{s_0,l_0,\dots,s_{i+1}}=l_{i+1}| \tilde X_{i+1}=s_{i+1},\dots,\tilde X_0=s_0\rp & = \prob\lp \alpha_i\pow{s_0,a_0,\dots,s_{i+1}}=l_{i+1}\rp\\
    & =  P_l\pow{s_i,l_{i-1},\dots,l_0,s_0} 
\end{align*}
and
\begin{align*}
    &  \prob\lp X_{s_i,\Tilde{N}_{s_i}^{(i,l_i)}+1}\pow{l_i}=s_{i+1}|\tilde X_i=s_i,\tilde a_{i}=l_i,\dots,\tilde a_0=l_0,\tilde X_0=s_0\rp\\
    &\quad \times \prob(\tilde X_i=s_i,\tilde a_i=l_i,\dots,\tilde X_0=s_0,\tilde a_0=l_0)\\
    &  \ =  \prob\lp X_{s_i,n}\pow{l_i}=s_{i+1}|\tilde X_i=s_i,\tilde a_{i}=l_i,\dots,\tilde a_0=l_0,\tilde X_0=s_0\rp\\
    &\quad \times \prob(\tilde X_i=s_i,\tilde a_i=l_i,\dots,\tilde X_0=s_0,\tilde a_0=l_0)\\
    &  \ =\prob\lp X_{s_i,n}\pow{l_i}=s_{i+1}\rp \prob(\tilde X_i=s_i,\tilde a_i=l_i,\dots,\tilde X_0=s_0,\tilde a_0=l_0)\\
    &\ = M_{s,t}\pow{l_i}\prob(\tilde X_i=s_i,\tilde a_i=l_i,\dots,\tilde X_0=s_0,\tilde a_0=l_0)\\
    &\ =M_{s,t}\pow{l_i}\prob(X_i=s_i, a_i=l_i,\dots,X_0=s_0, a_0=l_0),
\end{align*}
where the last equality follows by induction hypothesis. We finally get,
\begin{align*}
   & \prob\lp \tilde X_{i+1}=s_{i+1},\tilde a_{i+1}=l_{i+1},\dots,\tilde X_0=s_0,\tilde a_0=l_0 \rp \\
   & =  P_l\pow{s_i,l_{i-1},\dots,l_0,s_0} M_{s,t}\pow{l_i}\prob(X_i=s_i, a_i=l_i,\dots,X_0=s_0, a_0=l_0)
\end{align*}
which is the same as 
\begin{align*}
   &\prob\lp \tilde X_{i+1}=s_{i+1},\tilde a_{i+1}=l_{i+1},\dots,\tilde X_0=s_0,\tilde a_0=l_0 \rp \\ 
   &\ = \prob\lp  X_{i+1}=s_{i+1}, a_{i+1}=l_{i+1},\dots, X_0=s_0, a_0=l_0 \rp. 
\end{align*}
This completes the proof.
\end{proof}
Observe that Proposition \ref{prop:mod-sampling-scheme} does not require any assumption on the transition matrices $M$'s or the distribution of the actions $a_i$'s. However, it does require the finiteness structure of the CMC. It is unclear whether such an equivalent sampling schema exists for continuous state control processes and existing literature seems to rely on technically challenging adaptive estimation procedures (see Theorem 1 in \cite{banerjee2025adaptiveestimationtransitiondensity}). However, to the best of our knowledge, while adaptive estimation works well while deriving risk bounds \cite[etc.]{birge_model_2006,baraud_estimating_2009,baraud_rho-estimators_2018} and the techniques cannot be generalized to derive (functional) CLT's for the estimated transition densities. 
\subsection{Statement and Proof of Lemma~\ref{lemma:tight}}
\label{sec:lemma_tight}
\begin{lemma}\label{lemma:tight}
For each $(s,l,t) \in \mathcal S \times \mathcal A \times \mathcal S$,
    \begin{align*}
        \tilde \xi_{s,l,t} - \xi_{s,l,t} = \frac{\lp\tilde N_{s,t}\pow l - \lfloor \expec[N_s\pow l]\rfloor M_{s,t}\pow l \rp}{\sqrt{\expec[N_s\pow l]}} - \frac{\lp{ N_{s,t}\pow l-N_{s}\pow l M_{s,t}\pow l}\rp}{\sqrt{N_s\pow l}}\xrightarrow{\ p \ } 0.
    \end{align*}
\end{lemma}

\begin{proof}
    To show this, let $I\pow i:= \indicator[X_{s,i}\pow l=t]-M_{s,t}\pow l$ and define $S_n := \sum_{i=1}^n I\pow i$. We have
    \begin{align*}
       \frac{\lp\tilde N_{s,t}\pow l - \lfloor \expec[N_s \pow l]\rfloor M_{s,t}\pow l \rp}{\sqrt{n^{\frac{1}{1+\alpha}}}} - \frac{\lp{ N_{s,t}\pow l-N_{s}\pow l M_{s,t}\pow l}\rp}{\sqrt{n^{\frac{1}{1+\alpha}}}} & =  \frac{S_{\lfloor \expec[N_s \pow l] \rfloor}-S_{N_s \pow l}}{\sqrt{n^{\frac{1}{1+\alpha}}}}.
    \end{align*}
    Using Lemma \ref{lemma:consistency}, for all $\eps>0$ and $n$ sufficiently large, $\prob\lp \lv N_s \pow l -\lfloor \expec[N_s \pow l] \rfloor \rv> \sqrt{n^{\nicefrac{1}{1+\alpha}}}\epsilon\rp<\eps$. Therefore,
    \begin{align*}
    & \prob\lp \left| S_{\lfloor \expec[N_s \pow l] \rfloor} - S_{N_s \pow l} \right| > \sqrt{n^{\frac{1}{1+\alpha}}} \epsilon \rp \\
    \leq & \ \prob\lp \left| N_s \pow l - \lfloor \expec[N_s \pow l] \rfloor \right| > \sqrt{n^{\frac{1}{1+\alpha}}} \epsilon \rp
    \\ & \hspace{1in} + \prob\lp \left| N_s \pow l - \lfloor \expec[N_s \pow l] \rfloor \right| \leq \sqrt{n^{\frac{1}{1+\alpha}}} \epsilon, \left| S_{\lfloor \expec[N_s \pow l] \rfloor} - S_{N_s \pow l} \right| > \sqrt{n^{\frac{1}{1+\alpha}}} \epsilon \rp \\
    \leq & \ \epsilon + \prob\lp \left| N_s \pow l - \lfloor \expec[N_s \pow l] \rfloor \right| \leq \sqrt{n^{\frac{1}{1+\alpha}}} \epsilon, \left| S_{\lfloor \expec[N_s \pow l] \rfloor} - S_{N_s \pow l} \right| > \sqrt{n^{\frac{1}{1+\alpha}}} \epsilon \rp.
    \end{align*}
    Observe that $\indicator[X_{s,i}\pow a=t]$ and $M_{s,t}\pow a$ are between $0$ and $1$. Hence $\lv I \pow i\rv \leq 1$. Then 
    \[
    \prob\lp \left| N_s \pow l - \lfloor \expec[N_s \pow l] \rfloor \right| \leq \sqrt{n^{\nicefrac{1}{1+\alpha}}} \epsilon, \left| S_{\lfloor \expec[N_s \pow l] \rfloor} - S_{N_s \pow l} \right| > \sqrt{n^{\nicefrac{1}{1+\alpha}}} \epsilon \rp = 0,\] and 
    \[
    \prob\lp \left| S_{\lfloor \expec[N_s \pow l] \rfloor} - S_{N_s \pow l} \right| > \sqrt{n^{\nicefrac{1}{1+\alpha}}} \epsilon \rp \leq \epsilon.\] 
    Since $\epsilon$ is arbitrary, it now follows that
    \begin{align*}
        \frac{S_{\lfloor \expec[N_s \pow l] \rfloor}-S_{N_s \pow l}}{\sqrt{n^{\frac{1}{1+\alpha}}}}\xrightarrow{\ p\ }0.
    \end{align*}
    Next, observe that
    \begin{align*}
        & \frac{\lp\tilde N_{s,t}\pow l - \lfloor \expec[N_s \pow l]\rfloor M_{s,t}\pow l \rp}{\sqrt{\expec[N_s \pow l]}} - \frac{\lp{ N_{s,t}\pow l-N_{s}\pow l M_{s,t}\pow l}\rp}{\sqrt{N_s \pow l}}\\
        & \quad =\sqrt{\frac{n^{\frac{1}{1+\alpha}}}{\expec[N_s \pow l]}}\lp \frac{\lp\tilde N_{s,t}\pow l - \lfloor \expec[N_s \pow l]\rfloor M_{s,t}\pow l \rp}{\sqrt{n^{\frac{1}{1+\alpha}}}} - \frac{\lp{ N_{s,t}\pow l-N_{s}\pow l M_{s,t}\pow l}\rp}{\sqrt{n^{\frac{1}{1+\alpha}}}\lp\sqrt{{N_s \pow l}/{\expec[N_s \pow l]}}\rp}\rp.
    \end{align*}
    By Proposition~\ref{prop:count-lower-bound}, for sufficiently large $n$, $\expec[N_s \pow l] \geq \nicefrac{1}{c} n^{\nicefrac{1}{1+\alpha}}$, where $c > 0$ is a constant, and $\alpha \in (0,1)$. 
    By Lemma \ref{lemma:consistency}, $N_s \pow l/\expec[N_s \pow l]\xrightarrow{\ a.s. \ }1$. Therefore, $\sqrt{{n^{\nicefrac{1}{1+\alpha}}}/{\expec[N_s \pow l]}}\leq \sqrt{c}$ and 
    \begin{align*}
        & \frac{\lp\tilde N_{s,t}\pow l - \lfloor \expec[N_s \pow l]\rfloor M_{s,t}\pow l \rp}{\sqrt{\expec[N_s \pow l]}} - \frac{\lp{ N_{s,t}\pow l-N_{s}\pow l M_{s,t}\pow l}\rp}{\sqrt{N_s \pow l}} \\
        \leq & \ \sqrt{c}\lp \frac{\lp\tilde N_{s,t}\pow l - \lfloor \expec[N_s \pow l]\rfloor M_{s,t}\pow l \rp}{\sqrt{n^{\frac{1}{1+\alpha}}}} - \frac{\lp{ N_{s,t}\pow l-N_{s}\pow l M_{s,t}\pow l}\rp}{\sqrt{n^{\frac{1}{1+\alpha}}}\lp\sqrt{{N_s \pow l}/{\expec[N_s \pow l]}}\rp}\rp
        \xrightarrow{\ p \ }  0.
    \end{align*}
    This completes the proof.
\end{proof}

\section{Proof of Theorem \ref{thm:eta-clt}}
\label{sec:proof-eta-clt}
\begin{proof}
We begin by presenting the following generalization of the sampling scheme given in \cite{billingsley1961statistical} to the case of controlled Markov chains. 
For each $l \in \mathcal{A}$, create the following infinite array of i.i.d random variables which are also \emph{independent} of the data $\indexeddata$:
\begin{align}
    \Xbb^{(l)} : \left[ \begin{array}{ccccc}
        X_{1,1}^{(l)} & X_{1,2}^{(l)} & \dots & X_{1,\tau}^{(l)} & \dots \\
        X_{2,1}^{(l)} & X_{2,2}^{(l)} & \dots & X_{2,\tau}^{(l)} & \dots \\ 
         \dots &\dots &\dots &\dots & \dots \\
        X_{d,1}^{(l)} & X_{d,2}^{(l)} & \dots & X_{d,\tau}^{(l)} & \dots
    \end{array} \right],
    \nonumber
\end{align}
where, $\forall \ (s,t,\tau)\in \statespace \times \statespace \times \mathbb{N}$, the random variables $X_{s,\tau}^{(l)}$ follow the probability mass function given by $\prob (X_{s,\tau}^{(l)}=t)=M_{s,t}^{(l)}$.
Moreover, for every time $i\geq 1$, and $\lp(s_0,l_0),\dots,(s_{i-1},l_{i-1}),s_i\rp\in (\mathcal{S}\times\mathcal{A})^{i}\times \mathcal{S}$, let $\alpha_i\pow{(s_0,l_0),\dots,(s_{i-1},l_{i-1}),s_i}$ be independent random variables with support $\mathcal{A}$ and probability mass function given by
\begin{align*}
& P_l\pow{(s_0,l_0),\dots,(s_{i-1},l_{i-1}),s_i} \\
     := & \ \prob( \alpha_i\pow{(s_0,l_0),\dots,(s_{i-1},l_{i-1}),s_i}=l) \\
      = & \ \prob(a_i=l|X_i=s_i,\History_0^{i-1}=s_0,l_0,\dots,s_{i-1},l_{i-1}).
\end{align*}
The sampling scheme runs as follows. Set $(\tilde X_0,\tilde a_0)\overset{d}{=} (X_0,a_0)$ and recursively sample $\Tilde{X}_{i+1}= X_{\Tilde{X_i},\Tilde{N}_{\Tilde{X_i}}^{(i,\tilde a_i)}+1}\pow{\tilde a_i}$ from the array $\Xbb^{(\tilde a_{i})}$ and $\tilde a_{i+1}{=}\alpha_{i+1}\pow{\tilde X_0,\tilde{a_0},\dots,\tilde X_{i+1}}$, where each $i\geq 0$. 
Define $\tilde{N}_{s}\pow{i,l}:=\underset{j\leq i}{\sum}\indicator[\Tilde{X}_j=s,\tilde a_j=l]$ and $\tilde{N}_{s}\pow{l}:=\sum_{i=1}^{n} \indicator[\Tilde{X}_i=s,\tilde a_i=l]$, where $\tilde {N}_{s}\pow{n, l}=\tilde {N}_{s}\pow{l}$. 
This completes the sampling scheme.

Proposition~\ref{prop:mod-sampling-scheme} implies that $\tilde{N}_{s}\pow{l} \overset{d}{=} N_{s}\pow{l}$, hence $\expec[\tilde N_s\pow l]=\expec[N_s\pow l]$. 
Let $\tilde N_{s, t}\pow l$ be a random variable defined as 
\begin{equation}
\tilde N_{s, t}\pow l:=\sum_{\tau=1}^{\lfloor \expec[\tilde N_s\pow l] \rfloor} \mathbbm{1}\left[\tilde X_{i, \tau} \pow l=t \right].  
\end{equation}
Let the $kd^2$ dimensional vector $(\tilde \xi_{s,l,t})$ be defined element-wise as
\begin{align*}
\tilde \xi_{s,l,t}:= \frac{\lp\tilde N_{s,t}\pow l - \lfloor \expec[\tilde N_s\pow l]\rfloor M_{s,t}\pow l \rp}{\sqrt{\expec[\tilde N_s\pow l]}} = \frac{\lp\tilde N_{s,t}\pow l - \lfloor \expec[N_s\pow l]\rfloor M_{s,t}\pow l \rp}{\sqrt{\expec[N_s\pow l]}}.
\end{align*}
By definition, $\lp \tilde N_{s,1} \pow l, \ldots,   \tilde N_{s,d} \pow l\rp$ is the frequency count of $\{X_{s,1} \pow l, \ldots, X_{s,\lfloor \expec[N_s \pow l] \rfloor}\pow l\}$.
From the independence of $\{\mathbb X \pow l\}$ and the central limit theorem for multinomial trials, it follows that $\tilde \xi$ is distributed asymptotically normally with covariance matrix given by (\ref{eq:correlation}) and
    \begin{align*}
    \tilde \xi \overset{d}{\rightarrow} \Ncal(0,\Lambda).
    \end{align*}
Lemma~\ref{lemma:tight} implies that $\tilde \xi \xrightarrow{\ p \ } \xi$. Therefore, $\xi \overset{d}{\rightarrow} \Ncal(0,\Lambda)$. This completes the proof.

\end{proof}

\section{Proofs of CLTs for the Value, Q-, and Advantage Functions, and the Optimal Policy Value}
\subsection{Proof of Theorem \ref{thm:VQA}}\label{prf:VQA}
\begin{proof}
First, we show the asymptotic normality of $\sqrt{n} \lp  \hat V_\Pi - V_\Pi \rp$.
    Observe that 
    \begin{align*}
        \hat V_\Pi - V_\Pi = \lp\lp I-\alpha\Pi \mathbf{\hat M}\rp^{-1} - \lp I-\alpha\Pi \mathbf{M}\rp^{-1}\rp g,
    \end{align*}
    where the right-hand side can be rewritten as
    \begin{align*}
         \text{RHS} &= \alpha \lp I-\alpha\Pi \mathbf{\hat M}\rp^{-1} \Pi \lp \mathbf{\hat M} - \mathbf{M} \rp \lp I-\alpha\Pi \mathbf{M}\rp^{-1} g \\
         & = \alpha \lp I-\alpha\Pi \hat{\mathbf{M}}\rp^{-1} \Pi \lp \hat{\mathbf{M}} - \mathbf{M} \rp V_\Pi,
    \end{align*}
    using the matrix property $A^{-1} - B^{-1} = A^{-1}(B-A)B^{-1}$.
    Taking the transpose and vectorizing both sides, we get
    \begin{align*}
        \lp  \hat V_\Pi - V_\Pi \rp^\top & = \alpha V_\Pi^\top (\mathbf{\hat M}^\top - \mathbf{M}^\top) \Pi^\top \lp I-\alpha\mathbf{\hat M}^\top \Pi^\top \rp^{-1} \\ 
        \Vec{\lp \hat V_\Pi - V_\Pi \rp^\top} & = \alpha \left[\lp I-\alpha \Pi \mathbf{\hat M} \rp^{-1} \Pi \otimes  V_\Pi^\top\right]\Vec{\mathbf{\hat M}^\top - \mathbf{M}^\top} \\
        \hat V_\Pi - V_\Pi & = \alpha \left[\lp I-\alpha \Pi \mathbf{\hat M} \rp^{-1} \Pi \otimes  V_\Pi^\top\right]\Vec{\mathbf{\hat M}^\top - \mathbf{M}^\top}.
    \end{align*}
    Now, multiplying both sides by $\sqrt{n}$ and using Corollary \ref{cor:improper-clt}, the asymptotic normality of $\sqrt{n} \lp  \hat V_\Pi - V_\Pi \rp$ follows from the Slutsky's theorem. 

Second, we show the asymptotic normality of $\sqrt{n} \lp  \hat Q_\Pi - Q_\Pi \rp$. 
Similar to the value function, observe that 
    \begin{align*}
        \hat Q_\Pi - Q_\Pi = \lp\lp I-\alpha\mathbf{\hat M}\Pi \rp^{-1} - \lp I-\alpha\mathbf{M}\Pi \rp^{-1}\rp r,
    \end{align*}
    whose right-hand side can be rewritten as
    \begin{align*}
         \text{RHS} &= \alpha \lp I-\alpha\mathbf{\hat M}\Pi \rp^{-1}  \lp \mathbf{\hat M} - \mathbf{M} \rp \Pi\lp I-\alpha \mathbf{M}\Pi\rp^{-1} r \\
         & = \alpha \lp I-\alpha \mathbf{\hat M}\Pi\rp^{-1}  \lp \mathbf{\hat M} - \mathbf{M} \rp \Pi Q_\Pi.
    \end{align*}
    Taking the transpose and vectorizing both sides, we get
    \begin{align*}
        \lp  \hat Q_\Pi - Q_\Pi \rp^\top & = \alpha Q_\Pi^\top \Pi^\top (\mathbf{\hat M}^\top - \mathbf{M}^\top) \lp I-\alpha\Pi^\top \mathbf{\hat M}^\top  \rp^{-1} \\ 
        \Vec{\lp \hat Q_\Pi - Q_\Pi \rp^\top} & = \alpha \left[\lp I-\alpha \mathbf{\hat M} \Pi \rp^{-1}  \otimes  Q_\Pi^\top \Pi^\top \right]\Vec{\mathbf{\hat M}^\top - \mathbf{M}^\top} \\
        \hat Q_\Pi - Q_\Pi & = \alpha \left[\lp I-\alpha \mathbf{\hat M} \Pi \rp^{-1}  \otimes  Q_\Pi^\top \Pi^\top \right]\Vec{\mathbf{\hat M}^\top - \mathbf{M}^\top} .
    \end{align*}
    Now, multiplying both sides by $\sqrt{n}$ and using Corollary \ref{cor:improper-clt}, the asymptotic normality of $\sqrt{n} \lp  \hat Q_\Pi - Q_\Pi \rp$ follows from the Slutsky's theorem. 

Finally, we show the asymptotic normality of $\sqrt{n} \lp  \hat A_\Pi - A_\Pi \rp$.
Recall that $K = \operatorname{diag}(\mathbf{1}_k, \ldots, \mathbf{1}_k)$, where $\mathbf{1}_k$ is the ones vector of length $k$, and $A_\Pi = Q_\Pi - KV_\Pi$.
Then
\begin{align*}
\hat A_\Pi - A_\Pi & = \lp \hat Q_\Pi - Q_\Pi\rp - K \lp \hat V_\Pi - V_\Pi\rp \\
& = \alpha \lb \lp I - \alpha \mathbf{\hat M} \Pi\rp^{-1} \otimes Q_\Pi^\top \Pi^\top \rb \Vec{\mathbf{\hat M}^\top- \mathbf{M}^\top} \\
& \hspace{1in} -\alpha K \lb \lp I - \alpha \Pi \mathbf{\hat M} \rp^{-1} \Pi \otimes V_\Pi^\top  \rb \Vec{\mathbf{\hat M}^\top - \mathbf{M}^\top}.
\end{align*}
Now, multiplying both sides by $\sqrt{n}$ and using Corollary \ref{cor:improper-clt}, the asymptotic normality of $\sqrt{n} \lp  \hat A_\Pi - A_\Pi \rp$ follows from the Slutsky's theorem. 
\end{proof}
\subsection{Proof of Corollary~\ref{cor:Q-proper}}\label{sec:prf-Q-thm}
\begin{proof}
Following an entirely analogous way to the proof of Theorem~\ref{thm:VQA}, we get
\begin{align*}
    \hat Q_\Pi - Q_\Pi = B\Vec{\mathbf{\hat M}^\top - \mathbf{M}^\top},
\end{align*}
where $B=\alpha\left[\lp I-\alpha \mathbf{\hat M} \Pi \rp^{-1} \otimes  Q_\Pi^\top  \Pi^\top \right]$ is a $dk \times d^2k$ matrix.
Let $b_{(s,l),(s,t)}\pow l$ denote the $\lp (s-1)k+l, (s-1)dk+(l-1)d+t \rp$ entry of $B$.
Then
\begin{align*}
    \lp\hat Q_\Pi - Q_\Pi\rp \odot \Vec{\mathbf{N}^\top} &= B\Vec{\mathbf{\hat M}^\top - \mathbf{M}^\top} \\
    &=\left[\sum_{(s,l) \in \mathcal S \times \mathcal A, t \in \mathcal S} \sqrt{N_1 \pow 1} b_{(1,1), (s,t)} \pow l 
        \lp \hat M_{s,t} \pow l - M_{s,t} \pow l\rp, \right. \\
    &\quad \left. \ldots, \sum_{(s,l) \in \mathcal S \times \mathcal A, t \in \mathcal S} \sqrt{N_d \pow k} b_{(d,k), (s,t)} \pow l 
        \lp \hat M_{s,t} \pow l - M_{s,t} \pow l\rp\right]^\top \\
    &= \left[\sum_{(s,l) \in \mathcal S \times \mathcal A, t \in \mathcal S} \sqrt{\frac{N_1 \pow 1}{N_s \pow l}} b_{(1,1), (s,t)} \pow l 
        \sqrt{N_s \pow l} \lp \hat M_{s,t} \pow l - M_{s,t} \pow l\rp, \right. \\
    &\quad \left. \ldots, \sum_{(s,l) \in \mathcal S \times \mathcal A, t \in \mathcal S} \sqrt{\frac{N_d \pow k}{N_s \pow l}} b_{(d,k), (s,t)} \pow l 
        \sqrt{N_s \pow l}\lp \hat M_{s,t} \pow l - M_{s,t} \pow l\rp\right]^\top.
\end{align*}
Lemma~\ref{lemma:consistency} ensures that $\nicefrac{\expec[N_s \pow l]}{N_s \pow l} \xrightarrow{\ a.s. \ } 1$, and Assumption~\ref{ass:semi-ergodic} ensures that $\nicefrac{\expec[N_s \pow l]}{n} \rightarrow p_s \pow l > 0$, hence $\sqrt{\nicefrac{N_{s'} \pow {l'}}{N_s \pow l}} \xrightarrow{\ a.s. \ } \nicefrac{p_{s'} \pow {l'}}{p_s \pow l}$ for any $(s,l),(s',l') \in \mathcal S \times \mathcal A$.
The asymptotic normality of $\lp\hat Q_\Pi - Q_\Pi\rp \odot \Vec{\mathbf{N}^\top}$ then follows from Theorem~\ref{thm:eta-clt}.
\end{proof}
\subsection{Proof of Theorem \ref{thm:optpol}}\label{prf:optpol}
\begin{proof}
    Under the hypothesis of the theorem, as long as $\sup_{\Pi\in \Delta(\mathcal{A})}\|\hat V_\Pi-V_\Pi\|_{\infty}\overset{p}{\rightarrow}0$ it follows from the consistency of M-estimators \citep[Theorem 5.7]{vandervaart} that $\hat \Pi_{opt}\overset{p}{\rightarrow}\Pi_{opt}$. Fix any $\Pi$, and observe that
    \begin{align*}
        \left\|V_\Pi-\hat{V}_\Pi\right\|_{\infty} & =\left\|\lp(I-\alpha\Pi\hat{\mathbf{M}})^{-1}-(I-\alpha\Pi\mathbf{M})^{-1}\rp g\right\|_{\infty} \\
        & \leq\left\|\lp\lp I-\alpha\Pi\hat{\mathbf{M}}\rp^{-1}-\lp I-\alpha\Pi\mathbf{M}\rp^{-1}\rp\right\|_{\infty}\|g\|_{1} \\
        & \leq(1-\alpha)^{-2}\sqrt{dk}\left\|((I-\alpha\Pi\hat{\mathbf{M}})-(I-\alpha\Pi\mathbf{M}))\right\|_{\infty}\|g\|_{1} \\
        & \leq\frac{\alpha}{(1-\alpha)^{2}}\sqrt{dk}\left\|\Pi\hat{\mathbf{M}}-\Pi\mathbf{M}\right\|_{\infty}\|g\|_{1}\\
        & \leq \frac{\alpha}{(1-\alpha)^{2}}\sqrt{dk}\|\Pi\|_{1,\infty}\left\|\hat{\mathbf{M}}-\mathbf{M}\right\|_{\infty}\|g\|_{1}\\
        & = \frac{\alpha}{(1-\alpha)^{2}}\sqrt{dk}\left\|\hat{\mathbf{M}}-\mathbf{M}\right\|_{\infty}\|g\|_{1}
     \end{align*}
    where $\|\cdot\|_{1,\infty}$ is the row-wise $1$ and column-wise $\infty$ norm of a matrix. Observe that the rows of any policy sum to $1$, therefore $\|\Pi\|_{1,\infty}=1$. The right-hand side is independent of $\Pi$ and $\mathbf{\hat M}\overset{p}{\rightarrow}\mathbf{M}$. Therefore, by taking the supremum with respect to $\Pi$ on both sides and letting $n$ tend to $\infty$, an application of \citet[Theorem 5.7]{vandervaart} implies that $\hat \Pi_{opt}\overset{p}{\rightarrow}\Pi_{opt}$.

    We now prove the second part of our theorem. Since $\hat \Pi_{opt}$ maximizes $\hat V_{\hat \Pi_{opt}}$, it follows that for any given state $s$, $\hat V_{\hat \Pi_{opt}}(s)-\hat V_{\Pi_{opt}}(s)>0$, and
    \begin{align*}
        \hat V_{\hat \Pi_{opt}}-V_{\Pi_{opt}} & = \hat V_{\hat \Pi_{opt}}-\hat V_{\Pi_{opt}}+\hat V_{\Pi_{opt}}-V_{\Pi_{opt}}\\
        & \geq \hat V_{\Pi_{opt}}-V_{\Pi_{opt}},
    \end{align*}
    where the vector inequalities are coordinate-wise. It similarly follows that
    \begin{align*}
        \hat V_{\hat \Pi_{opt}}-V_{\Pi_{opt}} \leq \hat V_{\hat \Pi_{opt}}-V_{\hat \Pi_{opt}}. 
    \end{align*}
    Therefore, 
    \begin{align*}
         \hat V_{\Pi_{opt}}-V_{\Pi_{opt}}&\leq  \hat V_{\hat \Pi_{opt}}-V_{\Pi_{opt}}\leq \hat V_{\hat \Pi_{opt}}-V_{\hat \Pi_{opt}} \text{ and }\\
         \sqrt{n}\lp\hat V_{\Pi_{opt}}-V_{\Pi_{opt}}\rp&\leq  \sqrt{n}\lp\hat V_{\hat \Pi_{opt}}-V_{\Pi_{opt}}\rp\leq \sqrt{n}\lp \hat V_{\hat \Pi_{opt}}-V_{\hat \Pi_{opt}}\rp.
    \end{align*}
    Since $\hat \Pi_{opt}\overset{p}{\rightarrow}\Pi_{opt}$ it now follows using the Slutsky's theorem and the Sandwich theorem that
    \begin{align*}
         \sqrt{n}\lp\hat V_{\hat \Pi_{opt}}-V_{\Pi_{opt}}\rp\overset{d}{\rightarrow}\Ncal\lp 0,\Sigma_V\pow{\Pi_{opt}} \rp.
    \end{align*}
    This completes the proof.
\end{proof}
\section{Proofs of Other Propositions}
\subsection{Proof of Proposition~\ref{prop:mle}}
\label{sec:prf-mle}
\begin{proof}
    Given a sample of the controlled Markov chain $\{(s_i, l_i)\}_{i=0}^n$, the likelihood of the transition probabilities $\mathbf{M}$ is
    \[
    \mathcal L(\mathbf{M}) = \mathbb P(X_0 = s_0) \pi\pow 0(s_0, l_0) \prod_{i=1}^{n} \lp M_{s_{i-1}, s_i} \pow{l_{i-1}} \pi\pow i(s_i, l_i) \rp,
    \]
    where $\pi\pow i$ is the conditional probability of $\{X_i=s_i,a_i=l_i\}$ given $\{\History_0^{i-1}=\history_0^{i-1}\}$. Then,
    the log-likelihood is
    \[
    l(\mathbf{M}) = \sum_{i=1}^{n} \log M_{s_{i-1}, s_i} \pow {l_{i-1}}+ \sum_{i=1}^n \log \pi\pow i(s_i, l_i) + \log \mathbb P(X_0=s_0) + \log \pi\pow 0(s_0, l_0).
    \]
    Given the $dk$ constraint equations
    \[
    \sum_{t \in \mathcal S} M_{s, t} \pow l = 1,
    \]
    and $dk$ lagrange multipliers $\lambda_{s,l}, \, s=1, \ldots, d, \, l=1, \ldots, k$, the Lagrangian is
    \[
    l(\mathbf{M}) - \sum_{s=1}^d \sum_{l=1}^k \lambda_{s,l}\lp \sum_{t \in \mathcal S} M_{s, t} \pow l - 1 \rp.
    \]
    Taking derivatives with respect to $M_{s,t} \pow l$, and setting them to be $0$ at $\hat M_{s,t} \pow l$, we have
    \begin{align*}
        \frac{N_{s,t} \pow l}{\hat M_{s,t} \pow l} - \lambda_{s,l} &= 0, \\
        \frac{N_{s,t} \pow l}{\lambda_{s,l}} &= \hat M_{s,t} \pow l.
    \end{align*}
    Thus,
    \begin{align*}
        \sum_{t \in \mathcal S} \frac{N_{s,t} \pow l}{\lambda_{s,l}} = 1,\qquad  \lambda_{s,l} = \sum_{t \in \mathcal S} N_{s,t} \pow l = N_s \pow l.
    \end{align*}
    Therefore, $\hat M_{s,t} \pow l = \frac{N_{s,t} \pow l}{N_s \pow l}$, the count-based empirical estimator is the maximum likelihood estimator.
\end{proof}
\subsection{Proof of Proposition~\ref{prop:count-lower-bound}}
\label{sec:proof-count-lower-bound}
\begin{proof}
Define the process $Z_i$ as
\begin{align*}
Z_i = \sum_{p=1}^i \lp\frac{\tau_{s,l} \pow p}{T_p} - 1\rp, Z_0 = 0.
\end{align*}
Then 
\begin{align*}
\mathbb{E}\left[Z_i | \mathcal{F}_{\sum_{p=1}^{i-1} \tau_{s,l}\pow{p}}\right] &= \sum_{p=1}^{i} \mathbb{E}\left[\frac{\tau_{s,l} \pow p}{T_p} - 1| \mathcal{F}_{\sum_{p=1}^{i-1} \tau_{s,l}\pow{p}}\right] \\
&= \mathbb{E}\left[\frac{\tau_{s,l} \pow i}{T_i} - 1| \mathcal{F}_{\sum_{p=1}^{i-1} \tau_{s,l}\pow{p}}\right] + \sum_{p=1}^{i-1} \lp\frac{\tau_{s,l} \pow p}{T_p} - 1\rp \\
&= \mathbb{E}\left[\frac{\tau_{s,l} \pow i}{T_i} - 1| \mathcal{F}_{\sum_{p=1}^{i-1} \tau_{s,l}\pow{p}}\right] + Z_{i-1}.
\end{align*}
As $\mathbb{E}[\tau_{s,l} \pow i | \mathcal{F}_{\sum_{p=1}^{i-1} \tau_{s,l}\pow{p}}] \leq T_i$ almost surely, $\mathbb{E}\left[\nicefrac{\tau_{s,l} \pow i}{T_i} - 1| \mathcal{F}_{\sum_{p=1}^{i-1} \tau_{s,l}\pow{p}}\right] < 0$, hence $\{Z_i\}$ is a supermartingale with respect to $\mathcal{F}_{\sum_{p=1}^{i-1} \tau_{s,l}\pow{p}}$.

Let $N = N_s \pow l (n) + 1$.
As $N_s \pow l(n)$ is a valid stopping time, $N$ is a valid stopping time, and $\sum_{i=1}^N \tau_{s,l} \pow i > n$.
By Doob's Optional Stopping Theorem,
\begin{align*}
\mathbb{E}[Z_N] \leq \mathbb{E}[Z_0] = 0 \implies \mathbb{E}\left[\sum_{i=1}^N \frac{\tau_{s,l} \pow i}{T_i}\right] \leq \mathbb{E}[N].
\end{align*}

We may assume without loss of generality that $\{T_i\}$ is non-decreasing.
Indeed, the non-decreasing case is the only regime in which the stopping-time
argument below is needed; if $\{T_i\}$ is not eventually non-decreasing, the
claim follows by simpler deterministic bounds.

Since $\sum_{i=1}^N \tau_{s,l} \pow i > n$, and $\{T_i\}$ is a non-decreasing sequence, we have
\begin{align*}
\sum_{i=1}^N \frac{\tau_{s,l} \pow i}{T_i} \geq \frac{\sum_{i=1}^N \tau_{s,l} \pow i}{T_N} > \frac{n}{T_N} \implies \mathbb{E}\left[\frac{n}{T_N}\right] \leq \mathbb{E}[N].
\end{align*}
For the convex function $\phi(x) = n/x$, the Jensen's inequality gives
\begin{align*}
\mathbb{E}\left[\frac{n}{T_N}\right] \geq \frac{n}{\mathbb{E}[T_N]}.
\end{align*}
Therefore,
\begin{align*}
\mathbb{E}[N] \geq \mathbb{E}\left[\frac{n}{T_N}\right]    \geq \frac{n}{\mathbb{E}[T_N]}.
\end{align*}
As $T_i = O(i^\alpha)$, we have $T_N = O(N^\alpha)$, hence $\mathbb{E}[T_N] = O(\mathbb{E}[N^\alpha])$. 
For the concave function $\phi(x) = x^\alpha$ ($\alpha < 1$), Jensen's inequality gives
\begin{align*}
\mathbb{E}[N^\alpha] \leq \mathbb{E}[N]^\alpha.
\end{align*}
Hence, $\mathbb{E}[T_N] = O(\mathbb{E}[N]^\alpha)$ and thus
$
\mathbb{E}[N] = \Omega\lp \frac{n}{\mathbb{E}[N]^\alpha} \rp
$, 
which implies 
$\mathbb{E}[N]^{1+\alpha} = \Omega(n)$ and in turn $\mathbb{E}[N_s \pow l (n)] = \Omega\lp n^{\frac{1}{1+\alpha}}\rp.$
\end{proof}

\subsection{Proof of Proposition~\ref{thm:no-clt}}
\label{sec:proof-no-clt}
\begin{proof}
Let $\Mcal_{\mathcal S,\mathcal A}$ be the class of all probability measures over state-action pairs for a CMC with initial distribution $\initialD$. 
Any element $\Pcal\in \Mcal_{\mathcal S,\mathcal A}$ has an associated set of transition matrices $\lc M\pow 1,\dots,M\pow k\rc$ and a conditional distributions over the action $\{a_i\}$ conditional on the history until time $i$.  
Then the \emph{minimax risk} of an estimator $\hat M =(\hat M\pow 1,\dots, \hat M\pow k)$ of $M=(M\pow 1,\dots,M\pow k)$ is defined as
\[
\Rcal_m := \inf_{\mathbf{\hat M}}\prob\lp \sup_{l\in\mathcal A}\| \hat M\pow l-M\pow l \|_{\infty} >\eps \rp.
\]    
The proof now proceeds by creating a counterexample and then bounding from below its minimax risk.
Suppose for some $(s,l) \in \mathcal S \times \mathcal A$,
\[
\mathbb P\lp a_i=l | X_i = s\rp = \nicefrac{1}{(i+1)^2} \text{ for all } i \geq 0.
\]
It is easy to verify that $\sum_{i=1}^{\infty}\prob\lp X_i=s, a_i=l \rp < \infty$.
By the Borel-Cantelli lemma, as \[\sum_{i=1}^{\infty}\prob\lp X_i=s, a_i=l \rp < \infty,\] 
$N_s \pow l(\infty) < \infty$ 
almost surely. 
Therefore, there exists a constant $N$ such that $N_s \pow l (\infty) \leq N$ almost surely.

Suppose that $|\mathcal S| = 2(N+1)$.
Given any $\epsilon \in (0, \nicefrac{1}{2(N+1)})$, the transition probabilities from the state-action pair $(s,l)$ is given by 
\begin{align*}
M_s \pow l(\xi) & := \lb M_{s,1} \pow l(\xi), \ldots, M_{s,2(N+1)} \pow l(\xi)\rb \\
& = \lb \underbrace{2\epsilon \xi_1, \ldots, 2\epsilon \xi_{N+1}}_{N+1 \text{ times}}, \underbrace{\frac{1}{N+1} - 2\epsilon \xi_1, \ldots, \frac{1}{N+1} - 2\epsilon \xi_{N+1}}_{N+1 \text{ times}} \rb, \numberthis\label{eq:xi_kernel}
\end{align*}
where $\xi = \{\xi_1, \ldots, \xi_{N+1}\} \in \{0, 1\}^{N+1}$ are unknown parameters.
Therefore, to correctly estimate the transition matrix $M_s \pow l(\xi)$, we only need to correctly estimate $\xi$. 

By (\ref{eq:xi_kernel}), we have $\mathbb P \lp N_{s,1} \pow l = 0, N_{s,N+2} \pow l = 0\rp \geq \lp\nicefrac{N}{N+1}\rp^N \geq \nicefrac{1}{e} > \nicefrac{1}{3}$.
Now we observe that
\[
\sup_{l\in\mathcal A}\| \hat M\pow l-M \pow l \|_{\infty} \geq\nrm{\hat M_s\pow{l} - M_{s} \pow l(\xi)}_\infty.
\]
Therefore,
\begin{align*}
   \Rcal_m
    &\ \geq \inf_{\hat M\pow 1} \sup_{\xi\pow{1}\in\lc0,1\rc\pow{d/3}} \prob\lp{\nrm{\hat M_s\pow{l} - M_{s} \pow l(\xi)}_\infty > \epsilon}\rp\\
    & \ \geq \inf_{\hat M\pow 1} \sup_{\xi\pow{1}\in\lc0,1\rc\pow{d/3}} \prob\lp{\nrm{\hat M_s\pow{l} - M_{s} \pow l(\xi)}_\infty > \eps\gn N_{s,1} \pow l = 0, N_{s,N+2} \pow l = 0}\rp \\
    & \hspace{2in} \times \mathbb P \lp N_{s,1} \pow l = 0, N_{s,N+2} \pow l = 0\rp \\
    & \ > \frac{1}{3} \inf_{\hat M\pow 1} \sup_{\xi\pow{1}\in\lc0,1\rc\pow{d/3}} \prob\lp{\nrm{\hat M_s\pow{l} - M_{s} \pow l(\xi)}_\infty > \eps\gn N_{s,1} \pow l = 0, N_{s,N+2} \pow l = 0}\rp .
\end{align*}
We note that whenever $\xi\pow1\neq\xi\pow{2}\in
\lc 0, 1\rc^{N+1}$, we have $
\nrm{M_{s} \pow l (\xi \pow 1)  - M_{s} \pow l (\xi \pow 2)}_\infty = 2\eps.
$
For any estimate $\hat M_s \pow l$ of $M_s \pow l (\xi)$, define $\xi^\star$ to be the $\xi$ such that
$
\xi^\star = \arg \min_{\xi}\nrm{\hat M_s \pow l - M_{s} \pow l (\xi)}_\infty.
$
Then for $\xi \neq\xi^\star$ we have 
\begin{align*}
2\eps  = \nrm{M_{s} \pow l(\xi) - M_{s} \pow l(\xi^\star)} _\infty & \leq  \nrm{M_{s} \pow l(\xi) - \hat M_s \pow l }_\infty + \nrm{\hat M_s \pow l - M_{s} \pow l (\xi^\star)}_\infty \\
& \leq 2 \nrm{M_{s} \pow l(\xi) - \hat M_s \pow l }_\infty.
\end{align*}
Therefore, 
$\lc \xi^\star \neq \xi \rc \subset \lc  \nrm{M_{s} \pow l(\xi) - \hat M_s \pow l }_\infty\geq \eps \rc$ and $\mmrisk$ can be further lower bounded by 
\begin{align*}
  \mmrisk &> \frac{1}{3}
  \inf_{\hat M_s \pow l} \max_{\xi\in\lc0,1\rc^{N+1}}
  \prob\lp\xi^\star \neq \xi \gn N_{s,1} \pow l = 0, N_{s,N+2} \pow l = 0\rp\\
	&=
  \frac{1}{3} \inf_{\hat \xi} \max_{\xi\in\lc0,1\rc^{N+1}}
  \prob\lp\hat \xi \neq \xi \gn N_{s,1} \pow l = 0, N_{s,N+2} \pow l = 0\rp,
\end{align*}
where $\hat\xi$ is any estimate of $\xi^\star$ such that $\hat\xi: (X_0,a_0,\dots,X_n,a_n)\mapsto \{0,1\}^{N+1}$.
    
When $N_{s,1} \pow l = 0 \text{ and } N_{s,N+2} \pow l = 0$, the estimate $\hat \xi_1$ is equivalent to choosing uniformly over $\{0, 1\}$.
The probability of choosing incorrectly is $\nicefrac{1}{2}$. 
We get as a consequence that,
\begin{align*}
\frac{1}{3} \inf_{\hat \xi} \max_{\xi\in\lc0,1\rc^{N+1}}
  \prob\lp\hat \xi \neq \xi \gn N_{s,1} \pow l = 0, N_{s,N+2} \pow l = 0\rp \geq \frac{1}{6}.
\end{align*}
In conclusion, 
\[
    \Rcal_m > \frac{1}{6}.
\]
Therefore, for every $\epsilon \in (0, \nicefrac{1}{2(N+1)})$, there exists a controlled Markov chain such that it has no estimator $\hat M\pow l$ such that $\prob\lp  \sup_{l\in \mathcal A}\| \hat M\pow l-M\pow l \|_\infty>\eps\rp \leq \nicefrac{1}{6}$.

It follows that for any sequence $b_n\rightarrow\infty$, the event $\lc b_n\sup_{s,l,t} \lv \hat M_{s,t}\pow l-M_{s,t}\pow l\rv\rightarrow\infty\rc$ has probability at least $\nicefrac{1}{6}$. The conclusion follows.
\end{proof}

\subsection{Proof of Proposition~\ref{prop:gof}}
\label{sec:proof-GOF}
\begin{proof}
As Assumptions~\ref{ass:return-time-growth} and \ref{ass:eta-mix} hold, Theorem~\ref{thm:eta-clt} holds.
Then for any $(s,l) \in \mathcal S \times \mathcal A$, we have the following Pearson's chi-square test statistics
\[
    \sum_{t \in \mathcal S} \frac{\lp N_{s,t} \pow l - N_s \pow l M_{s,t} \pow l \rp^2}{N_s \pow l M_{s,t} \pow l}\overset{d}{\rightarrow} \chi^2_{d_{(s,l)} - 1},
\]
where $d_{(s,l)}$ is the number of states $t$ such that $M_{s,t} \pow l > 0$.
There are $dk$ such chi-square statistics in total.

Theorem~\ref{thm:eta-clt} implies that the $dk$ chi-square statistics are asymptotically independent.
Then
\begin{align*}
    \sum_{(s, l) \in \mathcal S \times \mathcal A, t \in \mathcal S} \frac{\lp N_{s,t} \pow l - N_s \pow l M_{s,t} \pow l \rp^2}{N_s \pow l M_{s,t} \pow l} \overset{d}{\rightarrow} \chi^2_{\sum_{(s,l) \in \mathcal S \times \mathcal A} d_{(s,l)} - dk}.
\end{align*}
\end{proof}

\bibliography{biblio_cleaned_minimal}

\end{document}